\documentclass{amsart}

\usepackage{amssymb}
\usepackage{thmtools,thm-restate}
\usepackage[nobysame]{amsrefs}
\usepackage{todonotes} 
\usepackage{mathpazo}
\usepackage{enumerate}
\usepackage{subcaption}
\usepackage{ifthen}
\usepackage{pifont}
\usepackage{longtable}
\usepackage{tikz,pgf,pgfmath}
	\usetikzlibrary{shapes,calc,decorations.pathmorphing,patterns}
	\pgfdeclarelayer{background}
	\pgfdeclarelayer{mid}
	\pgfsetlayers{background,mid,main}
	\definecolor{c1}{HTML}{FF9999}	
	\definecolor{c2}{HTML}{FFD27F}	
	\definecolor{c3}{HTML}{99C199}	
	\usepackage{xcolor}
	\colorlet{c4}{blue!40!gray!90}
	\colorlet{c5}{blue!60!green!70}
	\colorlet{c6}{gray!60!magenta!70}
	\colorlet{c7}{orange!50!gray}
	\colorlet{c8}{green!50!gray}
	\definecolor{red2}{rgb}{1.0, 0.01, 0.24}

	\newcommand{\Ao}{{\mathcal A}}

	\newcommand{\flats}{\mathsf{Flat}_{m,n}}
	
\usepackage{mathtools}
\usepackage[colorlinks=true,
	urlcolor=blue!20!black,
	citecolor=green!50!black]{hyperref}

\newtheorem{theorem}{Theorem}[section]
\newtheorem{proposition}[theorem]{Proposition}

\newtheorem{lemma}[theorem]{Lemma}
\newtheorem{corollary}[theorem]{Corollary}

\theoremstyle{remark}

\newtheorem{example}[theorem]{Example}
\newtheorem{remark}[theorem]{Remark}

\makeatletter
\@namedef{subjclassname@2020}{%
  \textup{2020} Mathematics Subject Classification}
\makeatother

\newcommand{\hm}[1]{\color{green!30!black}#1 \color{black}}
\newcommand{\defn}[1]{{\color{green!50!black}\emph{#1}}}
\newcommand{\defs}{\stackrel{\mathsf{def}}{=}}
\newcommand{\ie}{\text{i.e.},\;}

\newcommand{\Dyck}{\mathfrak{D}}

\newcommand{\valset}{\mathsf{Val}}
\newcommand{\val}{\mathsf{val}}
\newcommand{\retset}{\mathsf{Ret}}
\newcommand{\ret}{\mathsf{ret}}

\renewcommand{\dim}{\mathsf{dim}}
\newcommand{\codim}{\mathsf{codim}}
\newcommand{\swset}{\mathsf{Sw}}
\newcommand{\sswset}{\mathsf{SSw}}
\newcommand{\sw}{\mathsf{sw}}
\newcommand{\ssw}{\mathsf{ssw}}

\newcommand{\Schroder}{\mathfrak{S}}
\newcommand{\dgn}{\mathsf{dg}}

\newcommand{\cd}{\mathsf{cd}}

\newcommand{\NPol}{\mathcal{N}}

\newcommand{\APol}{\mathcal{A}}

\newcommand{\BPol}{\mathcal{B}}

\newcommand{\PPol}{\mathcal{P}}

\newcommand{\SPol}{\mathcal{S}}

\newcommand{\FPol}{\mathcal{F}}

\newcommand{\Cat}{\mathsf{Cat}}
\newcommand{\Nar}{\mathsf{Nar}}
\title{Refined Lattice Path Enumeration and Combinatorial Reciprocity}
\author{Henri M{\"u}hle}
\address{HM: Qoniac GmbH, Dr.-K{\"u}lz-Ring 15, 01067 Dresden, Germany.}
\email{henri.muehle@proton.me}
\author{Eleni Tzanaki}
\address{ET: Department of Mathematics and Applied Mathematics, University of Crete, 70013 Heraklion, Greece.}
\email{etzanaki@uoc.gr}


\BibSpec{collection.article}{%
    +{}  {\PrintAuthors}                {author}
    +{,} { \textit}                     {title}
    +{.} { }                            {part}
    +{:} { \textit}                     {subtitle}
    +{,} { \PrintContributions}         {contribution}
    +{,} { \PrintConference}            {conference}
    +{}  {\PrintBook}                   {book}
    +{,} { }                            {booktitle}
    +{}  { \PrintEditorsB}              {editor}
    +{,}  { }                           {publisher}
    +{,}  { }                           {address}
    +{,} { \PrintDateB}                 {date}
    +{,} { pp.~}                        {pages}
    +{,} { }                            {status}
    +{,} { \PrintDOI}                   {doi}
    +{,} { available at \eprint}        {eprint}
    +{}  { \parenthesize}               {language}
    +{}  { \PrintTranslation}           {translation}
    +{;} { \PrintReprint}               {reprint}
    +{.} { }                            {note}
    +{.} {}                             {transition}
    +{}  {\SentenceSpace \PrintReviews} {review}
}

\begin{document}

\allowdisplaybreaks

\begin{abstract}
	It is well known that the set of $m$-Dyck paths with a fixed height and a fixed amount of valleys is counted by the Fu{\ss}--Narayana numbers.  In this article, we consider the set of $m$-Dyck paths that start with at least $t$ north steps.  We give exact formulas for the number of such paths with fixed height, fixed number of returns and (i) fixed number of valleys, (ii) fixed number of valleys with $x$-coordinate divisible by $m$ and (iii) fixed number of valleys with $x$-coordinate not divisible by $m$.  The enumeration (ii) combinatorially realizes the $H$-triangle appearing in a recent article of Krattenthaler and the first author (Algebr. Comb. 5, 2022) in the context of certain parabolic noncrossing partitions.  Through a transformation formula due to Chapoton, we give an explicit formula for the associated $F$-triangle.  We realize this polynomial combinatorially by means of generalized Schr{\"o}der paths as well as flats in certain hyperplane arrangements.  Along the way we exhibit two new combinatorial reciprocity results.
\end{abstract}

\maketitle

\section{Introduction}

In the early 2000s, Chapoton introduced three bivariate polynomials attached to a finite Coxeter group $W$ that arose in his study of cluster algebras, cluster fans and associated toric varieties~\cite{chapoton04enumerative,chapoton06sur}.  The \defn{$F$-triangle} enumerates faces of the cluster complex with respect to positive and negative simple roots, the \defn{$H$-triangle} enumerates antichains in the root poset with respect to their size and the number of simple roots contained.  The \defn{$M$-triangle} can be seen as a bivariate generating function of the rank in the lattice of noncrossing partitions weighted by its M{\"o}bius function.  

Even though these polynomials are very different in nature, they are related in a quite remarkable fashion.  If $r$ denotes the rank of the Coxeter group $W$, then the following relations hold among them:
\begin{align}
	\tilde\FPol_{W}(x,y) & = x^{r}\tilde{\mathcal{H}}_{W}\left(\frac{x+1}{x},\frac{y+1}{x+1}\right)\label{eq:f_to_h},\\
	\tilde\FPol_{W}(x,y) & = y^{r}\tilde{\mathcal{M}}_{W}\left(\frac{y+1}{y-x},\frac{y-x}{y}\right)\label{eq:f_to_m}.
\end{align}
These relations are (equivalent to) \cite[Conjecture~6.1]{chapoton06sur} and \cite[Conjecture~1]{chapoton04enumerative}, respectively, and naturally induce a relation between $\tilde{\mathcal{H}}_{W}(x,y)$ and $\tilde  {\mathcal{M}}_{W}(x,y)$.  We refer to these relations as \defn{Chapoton's relations}.  

Equation~\eqref{eq:f_to_m} was proven uniformly by Athanasiadis in \cite{athanasiadis07some}, while \eqref{eq:f_to_h} was confirmed by Thiel in a more general setting~\cite{thiel14htriangle}.  In fact, Armstrong generalized the $F$-, $H$- and $M$-triangles to the Fu{\ss}--Catalan setting~\cite[Definition~5.3.1]{armstrong09generalized} as generating functions of appropriately chosen generalizations of the original objects.  Yet, Chapoton's relations remained to hold.  The relation between the $F$- and the $M$-triangle for this general setting was proven by Krattenthaler, M{\"u}ller and the second author~\cite{krattenthaler05mtriangle,krattenthaler06ftriangle,krattenthaler10decomposition,tzanaki08faces}.

In his thesis~\cite{williams13cataland}, Williams introduced generalizations of the cluster complex, the root poset and the lattice of noncrossing partitions to \defn{parabolic quotients} of Coxeter groups.  Combinatorially, these objects have been studied in \cite{ceballos20steep,fang21consecutive,fang21parabolic,muehle20ballot,muehle21noncrossing,muehle19tamari} for parabolic quotients of types $A$ and $B$.  In particular, the first author suggested generalizations of the $H$- and the $M$-triangle in parabolic type $A$ and conjectured for which quotients Chapoton's relations still remain true~\cite[Section~5]{muehle21noncrossing}.  An elementary proof of \eqref{eq:f_to_h} was given by Ceballos and the first author in the more general setting of $\nu$-Tamari lattices together with an appropriate definition of the $F$-triangle~\cite{ceballos21fh}.  

For certain well-behaved parabolic quotients of type $A$, noncrossing partition posets were studied from an enumerative point of view by Krattenthaler and the first author in \cite{krattenthaler22rank}.  These posets essentially depend on three positive integers $m,n,t$, and the following explicit formulas for the $F$-, $H$- and $M$-triangle were computed:
\begin{align}
	\tilde\FPol_{m,n,t}(x,y) & = \sum_{a=0}^{n-t}\sum_{b=0}^{n-t-a}\frac{t+b}{n}\binom{mn+a-1}{a}\binom{n}{t+a+b}x^{a}y^{b},\label{eq:ftriangle} \\
	\tilde{\mathcal{H}}_{m,n,t}(x,y) & = \sum_{a=0}^{n-t}\sum_{b=0}^{a}\Biggl(\binom{n-b-2}{a-b}\binom{mn-t+1}{n-t-a}\label{eq:htriangle}\\
		& \kern2cm -m\binom{n-b-1}{a-b}\binom{mn-t}{n-t-a-1}\Biggr)x^{a}y^{b},\notag\\
	\tilde{\mathcal{M}}_{m,n,t}(x,y) & = \sum_{a=0}\sum_{b=a}^{n-t}(-1)^{b-a}\frac{t(mn-t+1)-(n-t-s)(t-1)}{n(mn-t+1)}\\
		& \kern2cm \times\binom{n}{a}\binom{mn-t+1}{n-t-b}\binom{mn+b-a-1}{b-a}x^{a}y^{b}\notag.
\end{align}
The $M$-triangle was realized as a weighted rank-generating function of a certain poset of noncrossing partitions.  In \cite[Section~5]{krattenthaler22rank}, a combinatorial realization of the $H$-triangle was conjectured and proven for $m=1$.  The $F$-triangle remained mysterious. 

The purpose of this article is to give combinatorial interpretations of the polynomials $\tilde{\mathcal{H}}_{m,n,t}(x,y)$ and $\tilde\FPol_{m,n,t}(x,y)$.  We prove that they count certain lattice paths with respect to various statistics.  Let us fix some notation first.  An \defn{$(m,t)$-Dyck path} of height $n$ is a lattice path from $(0,0)$ to $(mn,n)$ with the following properties: (i) it uses only steps of the form $N=(0,1)$ (north steps) and $E=(1,0)$ (east steps), (ii) it stays weakly above the line $x=my$, (iii) it starts with at least $t$ north steps.  
We denote by $\Dyck_{m,n,t}$ the set of all such paths.  For $t=1$, we recover the set of $m$-Dyck paths enumerated for instance in \cite{duchon00enumeration,krattenthaler89counting}.

A \defn{valley} of $P\in\Dyck_{m,n,t}$ is a coordinate on $P$ which is preceded by an east step and followed by a north step.  We consider different types of valleys: (i) an \defn{$i$-valley} is a valley of $P$ whose $x$-coordinate is congruent to $i$ modulo $m$\footnote{For stylistic reasons, we use the term $m$-valleys rather than $0$-valleys.}, (ii) a \defn{return} is a valley of $P$ which is of the form $(mi,i)$.  

The first main result of this article is the following explicit enumeration.

\begin{theorem}\label{thm:valley_return_enumeration}
	Let $m,n,t>0$ be integers.  
	
	(i) The number of $(m,t)$-Dyck paths of height $n$ with exactly $a$-many valleys and $b$-many returns is
	\begin{equation}\label{eq:valley_form}
		\binom{n-t}{a}\binom{mn-b-1}{a-b} - m\binom{n-t+1}{a+1}\binom{mn-b-2}{a-b-1}.
	\end{equation}

	(ii) The number of $(m,t)$-Dyck paths of height $n$ with exactly $a$-many $m$-valleys and $b$-many returns is
	\begin{equation}\label{eq:mvalley_form}
		\binom{n-b-2}{a-b}\binom{mn-t+1}{n-t-a} - m\binom{n-b-1}{a-b}\binom{mn-t}{n-t-a-1}.
	\end{equation}
	
	(iii) The number of $(m,t)$-Dyck paths of height $n$ with exactly $a$-many $i$-valleys, for $i<m$, and $b$-many returns is
	\begin{equation}\label{eq:ivalley_form}
		\frac{t+b}{n}\binom{n}{a}\binom{mn-t-b-1}{n-t-a-b}.
	\end{equation}
\end{theorem}

We observe that for $i<m$, the enumeration of $i$-valleys does not depend on the remainder $i$.  We prove Theorem~\ref{thm:valley_return_enumeration} through generating functions.  Then, \eqref{eq:mvalley_form} represents the coefficient of $x^{a}y^{b}$ in $\tilde{\mathcal{H}}_{m,n,t}(x,y)$ from \eqref{eq:htriangle}.  

We wish to remark that several recent papers studied generalizations of $m$-Dyck paths with regard to enumeration of valleys with respect to different remainders modulo $m$~\cite{burstein20distribution,heuberger22enumeration}.  However, these generalizations are different from ours and do not consider a refined counting with respect to both valleys and returns.

The second main contribution of this article is the explicit enumeration of certain lattice paths with potential diagonal steps.  More precisely, a \defn{small $(m,t)$-Schr{\"o}der path} of height $n$ is a lattice path from $(0,0)$ to $(mn,n)$ with the following properties: (i) it uses only steps of the form $N=(0,1)$, $E=(1,0)$ or $D=(1,1)$ (diagonal steps), (ii) it stays weakly above the line $x=my$, (iii) it starts with at least $t$ north steps, (iv) it does not have a diagonal step starting on the line $x=my$ (returns are allowed).  A small $(m,t)$-Schr{\"o}der path is \defn{$m$-divisible} if every diagonal step ends at an $x$-coordinate divisible by $m$.  A diagonal step is \defn{cornered} if it connects coordinates of the form $(im-1,i)$ and $(im,i+1)$ for some $i$.  We obtain the following explicit formulas.

\begin{theorem}\label{thm:dimension_return_enumeration}
	Let $m,n,t>0$ be integers.
	
	(i)  The number of small $(m,t)$-Schr{\"o}der paths of height $n$ with exactly $b$-many returns and $n-t-a-b$ diagonal steps is
	\begin{equation}
		\frac{t+b}{n}\binom{mn+a-1}{a}\binom{mn}{n-t-a-b}.
	\end{equation}
	
	(ii)  The number of $m$-divisible small $(m,t)$-Schr{\"o}der paths of height $n$ with exactly $b$-many returns and $n-t-a-b$ diagonal steps is
	\begin{equation}\label{eq:ftriangle_coefficient}
		\frac{t+b}{n}\binom{mn+a-1}{a}\binom{n}{n-t-a-b}.
	\end{equation}
\end{theorem}

We observe that \eqref{eq:ftriangle_coefficient} represents the coefficient of $x^{a}y^{b}$ in $\FPol_{m,n,t}(x,y)$ from \eqref{eq:ftriangle}.  

\begin{theorem}\label{thm:diagonal_enumeration}
	Let $m,n,t>0$ be integers.
	
	(i)  The number of small $(m,t)$-Schr{\"o}der paths of height $n$ with exactly $a$-many diagonal steps and $b$-many cornered diagonal steps is
	\begin{multline}
		\sum_{i=1}^{n-t+1}\Biggl(\frac{(n-i)(mn-i+a+1)}{n(mn-i+a-b)} - \frac{n-t-i+1}{n-t+1} + \frac{(b+1)(m-1)}{mn-i+a-b}\Biggr)\\
			\times\binom{mn-i+a-b}{a-b}\binom{n-t+1}{n-t-i+1}\binom{mn}{i-a-1}.
	\end{multline}
	
	(ii)  The number of $m$-divisible small $(m,t)$-Schr{\"o}der paths of height $n$ with exactly $a$-many diagonal steps and $b$-many cornered diagonal steps is
	\begin{multline}
		\sum_{i=1}^{n-t+1}\Biggl(\frac{(n-i)(n-i+a+1)}{n(n-i+a-b)}-\frac{m(n-t-i+1)}{mn-t+1}\Biggr)\\
			\times\binom{n-i+a-b}{a-b}\binom{mn-t+1}{n-t-i+1}\binom{n}{i-a-1}.
	\end{multline}
\end{theorem}

\begin{theorem}\label{thm:positive_diagonal_enumeration}
	Let $m,n,t>0$ be integers.
	
	(i)  The number of positive $(m,t)$-Schr{\"o}der paths of height $n$ with exactly $a$-many diagonal steps and $b$-many cornered diagonal steps is
	\begin{equation}
		\Biggl(\frac{t+a}{n}-\frac{m(a-b)}{mn-b-1}\Biggr)\binom{mn-b-1}{a-b}\binom{mn+n-t-a-1}{n-t-a}.
	\end{equation}
	
	(ii)  The number of $m$-divisible positive $(m,t)$-Schr{\"o}der paths of height $n$ with exactly $a$-many diagonal steps and $b$-many cornered diagonal steps is
	\begin{equation}\label{eq:positive_form}
		\Biggl(\frac{t+a}{n}-\frac{a-b}{n-b-1}\Biggr)\binom{n-b-1}{a-b}\binom{mn+n-t-a-1}{n-t-a}.
	\end{equation}
\end{theorem}

Theorems~\ref{thm:dimension_return_enumeration}--\ref{thm:positive_diagonal_enumeration} are proven combinatorially.  More precisely, we consider the generating polynomials of the number arrays from Theorem~\ref{thm:valley_return_enumeration}(i) and (ii), and then transform $(m,t)$-Dyck paths with a distinguished set of valleys into small $(m,t)$-Schr{\"o}der paths.  This transformation can be recorded by the generating polynomials by variable substitutions.  The explicit formulas then follow by a direct computation.

It is straightforward to verify that whenever we plug in negative values for $m$ in any of the formulas  of Theorems~\ref{thm:valley_return_enumeration}--\ref{thm:positive_diagonal_enumeration} we obtain integers again.  This puts these formulas in the vicinity of the theory of combinatorial reciprocity as established in \cite{beck18combinatorial,stanley74combinatorial}.  Our main
results here are (i) Theorem~\ref{thm:dyck_schroder_reciprocity} which asserts that, for $t=1$, plugging in negative values for $m$ in \eqref{eq:mvalley_form} yields \eqref{eq:positive_form} up to sign, as well as (ii) Theorem~\ref{nar_ehrhart_reciprocity} which exhibits a new instance of Ehrhart reciprocity involving (positive) Fu{\ss}--Narayana numbers.  We do not have combinatorial explanations for the other reciprocities.

We conclude our paper with a combinatorial interpretation of equation \eqref{eq:f_to_h} in the setting of antichains in the root poset. More precisely, the antichains in the root poset are in bijection with the dominant regions of the corresponding  Catalan arrangement and the $H$-triangle counts regions according to a pair of statistics known as  separating and simple separating walls. Then, one can see the relation  in \eqref{eq:f_to_h} as a refined way of 
transforming the $h$-vector to the $f$-vector on the set of certain flats in the dominant regions of the Catalan arrangement. 

This article is organized as follows.  In Section~\ref{sec:path_counting}, we prove Theorem~\ref{thm:valley_return_enumeration}.  To achieve this, we formally define $(m,t)$-Dyck paths and the relevant statistics and set up the generating function machinery.  The actual generating function proofs are rather technical and are given in an appendix.    In Section~\ref{sec:schroder_paths} we introduce $(m,t)$-Schr{\"o}der paths and derive Theorems~\ref{thm:dimension_return_enumeration}--\ref{thm:positive_diagonal_enumeration} combinatorially through certain transformations of $(m,t)$-Dyck paths.  We discuss combinatorial reciprocity relations that appear in the context of our work in Section~\ref{sec:poly_reciprocity}.  In Section~\ref{sec:arrangements} we explain how the polynomials $\tilde\FPol_{m,n,t}(x,y)$ and $\tilde{\mathcal{H}}_{m,n,t}(x,y)$ from \eqref{eq:ftriangle} and \eqref{eq:htriangle} appear as counting polynomials of certain regions in the $m$-Shi arrangement.

\section{Enumeration of certain Dyck paths with respect to their number of valleys and returns}
	\label{sec:path_counting}

\subsection{Northeast paths}
\label{sec:northeast_paths}
A \defn{northeast path} is a lattice path starting at the origin and using only steps of type $N\defs(0,1)$ (\defn{north steps}) and $E\defs(1,0)$ (\defn{east steps}).  The empty path is denoted by $\epsilon$.

An \defn{$(m,t)$-Dyck path} of height $n$ is a northeast path from $(0,0)$ to $(mn,n)$ which starts with at  least $t$ north steps and never goes below the line $x=my$.  We write $\Dyck_{m,n,t}$ for the set of all $(m,t)$-Dyck paths of height $n$.  We also wish to express the fact that a northeast path $P$ uses exactly $n$ north steps by writing $\lvert P\rvert=n$.

\begin{lemma}\label{lem:mt_dyck_size}
	For $m,n,t>0$, the cardinality of $\Dyck_{m,n,t}$ is $\frac{mt+1}{mn+1}\binom{(m+1)n-t}{n-t}$.
\end{lemma}
\begin{proof}
	This follows from \cite[Theorem~10.4.5]{krattenthaler15lattice}.
\end{proof}

If $P$ is a northeast path, then a \defn{valley} of $P$ is a coordinate on $P$ which is preceded by an east step and followed by a north step.  We denote by $\val(P)$ the number of valleys of $P$.  

For $P\in\Dyck_{m,n,t}$, an \defn{$i$-valley} is a valley whose $x$-coordinate is equal to $i \pmod{m}$ for some $i\in[m]\defs\{1,2,\ldots,m\}$\footnote{For stylistic reasons, we wish to refer to valleys whose $x$-coordinate is divisible by $m$ as $m$-valleys rather than $0$-valleys.}.  We denote by $\val(P,i)$ the number of $i$-valleys of $P$.  A \defn{return} of $P$ is a valley of the form $(mi,i)$ for $i\in[n-1]$.  In particular, each return is an $m$-valley.  We denote by $\ret(P)$ the number of returns of $P$.

Let us define the set of all $(m,t)$-Dyck paths of height $n$ without returns by $\Dyck_{m,n,t}^{+}$, and consider the following collections:
\begin{align*}
	\Dyck_{m,\bullet,t} \defs \biguplus_{n\geq 0}\Dyck_{m,n,t},\\
	\Dyck_{m,\bullet,t}^{+} \defs \biguplus_{n\geq 0}\Dyck_{m,n,t}^{+}.
\end{align*}
The main contribution of this article is the explicit computation of the coefficients of the following generating functions:
\begin{align}
	D^{(t)}(x,y;z) & \defs \sum_{P\in\Dyck_{m,\bullet,t}}x^{\val(P)}y^{\ret(P)}z^{\lvert P\rvert},\label{eq:dyck_enumerator}\\
	C^{(t)}_{i}(x,y;z) & \defs \sum_{P\in\Dyck_{m,\bullet,t}}x^{\val(P,i)}y^{\ret(P)}z^{\lvert P\rvert}\label{eq:idyck_enumerator}.
\end{align}

In order to explicitly enumerate the number of $(m,t)$-Dyck paths of height $n$ with respect to the number of valleys and returns, we use a generating function approach.  Our main tool will be the following Lagrange--B{\"u}rmann formula.  If $F(z)=\sum_{n\geq 0}a_{n}z^{n}$ is a formal power series, then we write $\langle z^{n}\rangle\defs a_{n}$ for the coefficient of $z^{n}$.

\begin{lemma}[{\cite[Theorem~1.9b]{henrici74applied}}]\label{lem:lagrange_burmann}
	Let $f(z)$ and $g(z)$ be formal power series with $f(0)=0$ and let $F(z)$ be the compositional inverse of $f(z)$.  Then, for all non-zero integers $a$, 
	\begin{displaymath}
		\langle z^{a}\rangle g\bigl(F(z)\bigr) = \frac{1}{a}\langle z^{-1}\rangle g'(z)\bigl(f(z)\bigr)^{-a}.
	\end{displaymath}
\end{lemma}

\subsection{Valley enumeration for $(m,t)$-Dyck paths}
\label{sec:valley_mnt}
As a warm-up, we start with the enumeration of $(m,t)$-Dyck paths with respect to valleys and returns.  In other words, we consider the following counting polynomial:
\begin{equation}\label{eq:mnt_valley_return_enumerator}
	\NPol_{m,n,t}(x,y) \defs \sum_{P\in\Dyck_{m,n,t}}x^{\val(P)}y^{\ret(P)}.
\end{equation}

\subsubsection{The case $t=1$ without distinguishing the returns}
	\label{sec:mvalley_all}
Let us start with the case $t=1$, and pick $P\in\Dyck_{m,\bullet,1}$.  Then, either $P$ is empty or it can be decomposed as follows:
\begin{equation}\label{eq:basic_decomposition}
	P = NP_{0}EP_{1}E\cdots P_{m-1}EP_{m},
\end{equation}
where $P_{i}\in\Dyck_{m,\bullet,1}$.  The paths $P_{i}$ can possibly be empty, however the sum of  their heights should
be the height of $P$ reduced by one, i.e., $|P_0|+\cdots+|P_m|=|P|-1$.  See Figure~\ref{fig:basic_decomposition} for an illustration.

\begin{figure}
	\centering
	\includegraphics[page=1,scale=1]{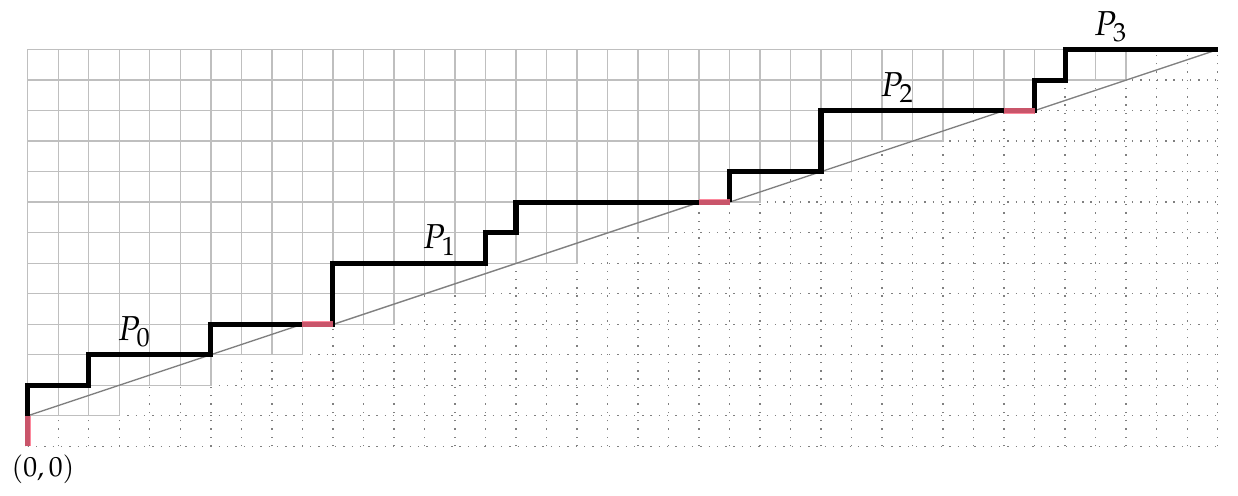}
	\caption{Illustrating the decomposition from \eqref{eq:basic_decomposition}.}
	\label{fig:basic_decomposition}
\end{figure}

Any valley of $P$ is either a valley of any of the $P_{i}$ or it is the starting coordinate of some $P_{i}\neq\epsilon$ for $i>0$.  This gives the following functional equation for $D^{(1)}(x,1;z)$:
\begin{equation}\label{eq:basic_functional_equation}
	D^{(1)}(x,1;z) = 1 + zD^{(1)}(x,1;z)\Bigl(1+x\Bigl(D^{(1)}(x,1;z)-1\Bigl)\Bigr)^{m},
\end{equation}
where the term $z$ in front of the second summand accounts for the fact that $\lvert P_0\rvert+\cdots+\lvert P_m\rvert=\lvert P\rvert-1$ and each term $\Bigl(1+x\Bigl(D^{(1)}(x,1;z)-1\Bigl)\Bigr)$ incorporates the fact that each $P_i$ with  $i\geq 1$ contributes no valley at all when empty, whereas it contributes one more valley (the one by which it is preceded) otherwise.

Let us write
\begin{equation}\label{eq:dyck_enumerator_shift}
	F^{(t)}(x,y;z)\defs D^{(t)}(x,y;z)-1.
\end{equation}
Then, by using the function $g_{\bullet}(z)\defs(z+1)(xz+1)^{m}$, we get
\begin{equation}\label{eq:dyck_val_gf_A}
	F^{(1)}(x,1;z) = zg_{\bullet}\bigl(F^{(1)}(x,1;z)\bigr).
\end{equation}
As a consequence, we find that the compositional inverse of $F^{(1)}(x,1;z)$ is
\begin{equation}\label{eq:dyck_val_gf_B}
	f(z) \defs \frac{z}{(z+1)(xz+1)^{m}}.
\end{equation}
We denote by $\langle z^{n}\rangle F^{(1)}(x,1;z)$ the coefficient of $z^{n}$ in $F^{(1)}(x,1;z)$.  We may thus write
\begin{align}\label{eq:narayana_def}
	\NPol_{m,n,1}(x,1) = \langle z^{n}\rangle F^{(1)}(x,1;z).
\end{align}
for the valley enumerator of $\Dyck_{m,n,1}$. For instructive purposes (of the things to come) we give a generating function proof here, and remark that Cigler's original proof is combinatorial.

\begin{proposition}[{\cite[Theorem~2]{cigler87some}}]\label{prop:m_dyck_valley}
	For $m,n>0$, we have
	\begin{equation}
		\NPol_{m,n,1}(x,1) = \sum_{r=0}^{n-1}\frac{1}{n}\binom{n}{r+1}\binom{mn}{r}x^{r}.
	\end{equation}
\end{proposition}
\begin{proof}
	We obtain the desired result by applying Lemma~\ref{lem:lagrange_burmann} to the functions $f(z)$ and $g_{\bullet}(z)$.  This proof is prototypical for the proofs to come, so we state it in detail.  By definition, we have
	\begin{align*}
		\NPol_{m,n,1}(x,1) & \overset{\parbox{1cm}{\centering\tiny\eqref{eq:narayana_def}}}{=} \langle z^{n}\rangle F^{(1)}(x,1;z)\\
		& \overset{\parbox{1cm}{\centering\tiny\eqref{eq:dyck_val_gf_A}}}{=} \langle z^{n-1}\rangle g_{\bullet}\bigl(F^{(1)}(x,1;z)\bigr)\\
		& \overset{\parbox{1cm}{\centering\tiny\text{Lem.~}\ref{lem:lagrange_burmann}}}{=} \frac{1}{n-1}\langle z^{-1}\rangle g'_{\bullet}(z)\bigl(f(z)\bigr)^{-(n-1)}\\
		& \overset{\parbox{1cm}{~}}{=} \frac{1}{n-1}\langle z^{-1}\rangle\Biggl( (xz+1)^{m-1}\Bigl(xz+1+mx(z+1)\Bigr)\\
		& \kern5cm \times\left(\frac{z}{(z+1)(xz+1)^{m}}\right)^{-(n-1)}\Biggr)\\
		& \overset{\parbox{1cm}{~}}{=} \frac{1}{n-1}\langle z^{n-2}\rangle\Biggl( (z+1)^{n-1}(xz+1)^{mn} + mx(z+1)^{n}(xz+1)^{mn-1}\Biggr)\\
		& \overset{\parbox{1cm}{~}}{=} \frac{1}{n-1}\langle z^{n-2}\rangle\Biggl( \sum_{k=0}^{n-1}\sum_{\ell=0}^{mn}\binom{n-1}{k}\binom{mn}{\ell}x^{\ell}z^{k+\ell} \\
		& \kern5cm + m\sum_{k=0}^{n}\sum_{\ell=0}^{mn-1}\binom{n}{k}\binom{mn-1}{\ell}x^{\ell+1}z^{k+\ell}\Biggr)\\
		& \overset{\parbox{1cm}{~}}{=} \frac{1}{n-1}\Biggl( \sum_{r=0}^{n-2}\binom{n-1}{n-2-r}\binom{mn}{r}x^{r} \\
		& \kern5cm + m\sum_{r=0}^{n-2}\binom{n}{n-2-r}\binom{mn-1}{r}x^{r+1}\Biggr)\\
		& \overset{\parbox{1cm}{~}}{=} \frac{1}{n-1}\Biggl( \sum_{r=0}^{n-2}\binom{n-1}{r+1}\binom{mn}{r}x^{r} + m\sum_{r=1}^{n-1}\binom{n}{r+1}\binom{mn-1}{r-1}x^{r}\Biggr)\\
		& \overset{\parbox{1cm}{~}}{=} \frac{1}{n-1}\sum_{r=0}^{n-1}\Biggl( \binom{n-1}{r+1}\binom{mn}{r} + m\binom{n}{r+1}\binom{mn-1}{r-1} \Biggr) x^{r}\\
		& \overset{\parbox{1cm}{~}}{=} \frac{1}{n-1}\sum_{r=0}^{n-1}\Biggl( \frac{n-r-1}{n}\binom{n}{r+1}\binom{mn}{r} + \frac{mr}{mn}\binom{n}{r+1}\binom{mn}{r} \Biggr) x^{r}\\
		& \overset{\parbox{1cm}{~}}{=} \sum_{r=0}^{n-1}\frac{1}{n}\binom{n}{r+1}\binom{mn}{r}x^{r}.\qedhere
	\end{align*}
\end{proof}

A combinatorial proof of Proposition~\ref{prop:m_dyck_valley} using the lattice-path perspective is given in \cite[Theorem~10.14.4]{krattenthaler15lattice} when setting $(c,d)=(mn,n)$ in \cite[Relation~(10.130)]{krattenthaler15lattice}.

\begin{remark}\label{rem:reverse_fuss_narayana}
	There are various other combinatorial interpretations of the coefficients appearing in Proposition~\ref{prop:m_dyck_valley}.  For instance, they enumerate $m$-divisible noncrossing partitions of $\{1,2,\ldots,mn\}$ with exactly  $r+1$ blocks \cite[Lemma~4.1]{edelman80chain}.
	
	Moreover, they give the entries of the $h$-vector of the $m$-cluster complex associated with the symmetric group of degree $n$;~\cite[Theorem~10.2]{fr-gcccc-05} and \cite[Corollary~1.3]{tz-pdgcc-06}. 
	
	By \cite[Theorem~1.2]{Ath_RGCN05}, these coefficients also count the number of dominant regions in the type-$A$ $m$-Catalan arrangement according to the number of separating walls\footnote{A \emph{separating wall} is a supporting hyperplane that separates a given region from the origin.} of type $x_i-x_j=m$ (see also Section \ref{sec:arrangements}).  These interpretations justify to call the coefficients of $\NPol_{m,n,1}(x,1)$ the \defn{(reverse) Fu{\ss}--Narayana numbers}.
\end{remark}
	
\begin{remark}\label{rem:fuss_narayana}
	Armstrong has defined the Fu{\ss}--Narayana polynomial in \cite[Definition~3.5.4]{armstrong09generalized} as the rank-generating polynomial in the poset of $m$-divisible noncrossing partitions of $\{1,2,\ldots,mn\}$.  The coefficient sequence of this polynomial is simply that of $\NPol_{m,n,1}(x,1)$ reversed.  This is the reason why we call the coefficients of $\NPol_{m,n,1}(x,1)$ the \emph{reverse} Fu{\ss}--Narayana numbers.  We will encounter Armstrong's Fu{\ss}--Narayana polynomial somewhat surprisingly in Proposition~\ref{prop:m_dyck_mvalley}.
\end{remark}

\subsubsection{The case $t=1$ with no returns}
	\label{subsec:no returns}

As a second step in the preparation of the computation of $D^{(t)}(x,y;z)$ we enumerate the $(m,1)$-Dyck paths without returns.  In other words, we are interested in the following generating function:
\begin{equation}\label{eq:positive_functional_equation}
	D^{(1)}(x,0;z) = \sum_{P\in\Dyck_{m,\bullet,1}^{+}}x^{\val(P)}z^{\lvert P\rvert}.
\end{equation}
If $P\in\Dyck_{m,\bullet,1}^{+}$, then it may be decomposed as in \eqref{eq:basic_decomposition} with $P_{m}=\epsilon$.  This gives the following functional equation:
\begin{equation}\label{eq:dyck_val_noret_gf_A}
	D^{(1)}(x,0;z) = 1 + zD^{(1)}(x,1;z)\Bigl(1 + x\Bigl(D^{(1)}(x,1;z)-1\Bigr)\Bigr)^{m-1}.
\end{equation}
Using $F^{(t)}(x,y;z)$ defined in \eqref{eq:dyck_enumerator_shift}, we get
\begin{align}
	\NPol_{m,n,1}(x,0) = \langle z^{n}\rangle F^{(1)}(x,0;z).
\end{align}

We obtain the following result, whose proof is essentially analogous to the one presented in Proposition~\ref{prop:m_dyck_valley}

\begin{restatable}{proposition}{nposform}\label{prop:m_dyck_valley_noret}
	For $m,n>0$, we have
	\begin{equation}
		\NPol_{m,n,1}(x,0) = \sum_{r=0}^{n-1}\frac{1}{n}\binom{n}{r+1}\binom{mn-2}{r}x^{r}.
	\end{equation}
\end{restatable}
\begin{proof}[Short proof]
	In view of \eqref{eq:dyck_val_noret_gf_A}, we may define $g_{+}(z)\defs(z+1)(xz+1)^{m-1}$ to obtain
	\begin{equation}
		F^{(1)}(x,0;z) = zg_{+}\bigl(F^{(1)}(x,1;z)\bigr).
	\end{equation}
	Then applying Lemma~\ref{lem:lagrange_burmann} to the functions $f(z)$ and $g_{+}(z)$ yields the result.  We provide the details in Appendix~\ref{app:prop_m_dyck_valley_noret}.
\end{proof}

\begin{remark}\label{rem:positive_narayana}
	If we sum up the coefficients of $\NPol_{m,n,1}(x,1)$, then we obtain
	\begin{displaymath}
		\Cat(m,n) \defs \frac{1}{mn+1}\binom{(m+1)n}{n} = \sum_{r=0}^{n-1}\frac{1}{n}\binom{n}{r+1}\binom{mn}{r},
	\end{displaymath}
	which is the well-known \defn{Fu{\ss}--Catalan number}.  If we sum up the coefficients of $\NPol_{m,n,1}(x,0)$, then we obtain
	\begin{displaymath}
		\Cat^{+}(m,n) \defs \frac{1}{mn-1}\binom{(m+1)n-2}{n} = \sum_{r=0}^{n-1}\frac{1}{n}\binom{n}{r+1}\binom{mn-2}{r},
	\end{displaymath}
	which is the so-called \defn{positive Fuss--Catalan number} that perhaps first appeared in \cite[Proposition~3.9]{fomin03ysystems}.  In that spirit, we may call the coefficients of $\NPol_{m,n,1}(x,0)$ the \defn{(reverse) positive Fu{\ss}--Narayana numbers}.
	
	Surprisingly little is known about combinatorial interpretations of the (reverse) positive Fu{\ss}--Narayana numbers.  They count the $h$-vector of the \emph{positive} $m$-cluster complex \cite[Corollary~5.5]{atz-pc-06}.  In reverse order, they count the number of \emph{bounded} dominant regions in the type-$A$ Catalan arrangement according to the number of their \emph{non}-separating walls of type $x_i-x_j=m$; \cite[Corollary~4.4]{atz-pc-06}.
	
	We are currently unaware of an explanation of these numbers in terms of $m$-divisible noncrossing partitions.
\end{remark}

\subsubsection{The case $t>1$}
Now we have gathered all the prerequisites to compute the coefficients of the generating function $D^{(t)}(x,y;z)$.  From that, we conclude the proof of Theorem~\ref{thm:mnt_narayana} which gives an explicit formula for the coefficients of $\NPol_{m,n,t}(x,y)$ from \eqref{eq:mnt_valley_return_enumerator}.

Let $P\in\Dyck_{m,\bullet,t}$ with $P\neq\epsilon$ and $\ret(P)=s$.  Since $P$ starts with at least $t$ north steps, we may decompose it as
\begin{equation}\label{eq:full_decomposition}
	P = N^{t}R_{1}R_{2}\cdots R_{t}Q_{1}Q_{2}\cdots Q_{s},
\end{equation}
where each $R_j$, when preceded by an extra north step, is a path in $\Dyck_{m,\bullet,1}^{+}$ and each $Q_{i}\in\Dyck_{m,\bullet,1}^{+}\setminus\{\epsilon\}$. See Figure~\ref{fig:general} for an illustration. 
In order to read off the valleys and return points of the original path $P$ from  \eqref{eq:full_decomposition}, we need to further decompose each $R_j$   as $R_{j}=P_{0,j}EP_{1,j}E\cdots P_{m-1,j}E$ with $P_{i,j}\in\Dyck_{m,\bullet,1}$, like we did in the case of Dyck paths without return points. 

\begin{figure}
	\includegraphics[page=2,width=\textwidth]{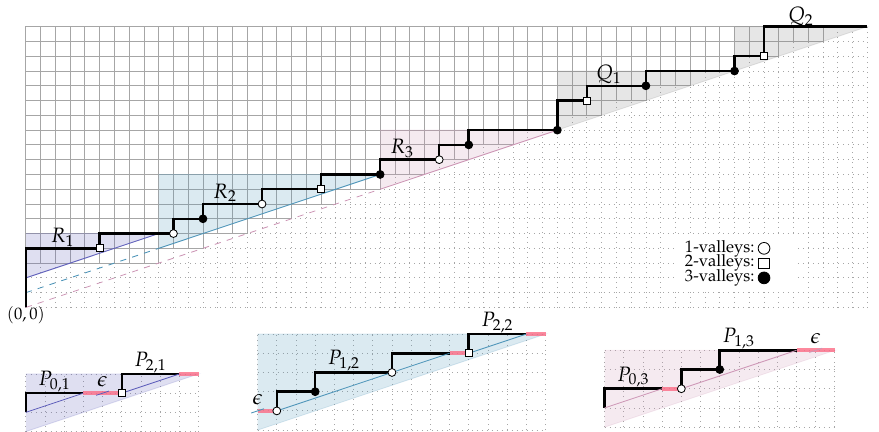}
	\caption{Illustrating the decomposition from \eqref{eq:full_decomposition}.}
	\label{fig:general}
\end{figure}

Then, the first coordinate of each $Q_{i}$ contributes both a valley and a return, because $Q_{i}\neq\epsilon$ by design.  The first coordinate of $P_{i,j}$ contributes another valley, except for $P_{0,1}$.  This gives the following functional equation:
\begin{multline}\label{eq:full_functional_equation}
	D^{(t)}(x,y;z) = 1 + \sum_{s=0}^{\infty}z^{t}
	\underbrace{
		D^{(1)}(x,1;z)\Bigl(1+x\bigl(D^{(1)}(x,1;z)-1\bigr)\Bigr)^{m-1}}_{\text{(a)}}\\
	\underbrace{	\Bigl(1+x\bigl(D^{(1)}(x,1;z)-1\bigr)\Bigr)^{m(t-1)}}_{\text{(b)}}\;
	\underbrace{\Bigl(xy\bigl(D^{(1)}(x,0;z)-1\bigr)\Bigr)^{s}}_{\text{(c)}}.
\end{multline}

Let us explain \eqref{eq:full_functional_equation} in more detail.\\
Part (a) corresponds to the subpath $R_1$. More precisely, the initial coordinate of $P_{0,1}$ does not contribute a valley; hence we pick up the generating function $D^{(1)}(x,1;z)$.  The initial coordinate of each $P_{i,1}$, for $i>0$, contributes a valley only if $P_{i,1}$ is non-empty; for each $i>0$, we thus get a factor $\bigl(1+x\big(D^{(1)}(x,1;z)-1\bigr)\big)$.\\
Part (b) corresponds to $R_2,\ldots,R_t$. The initial coordinate of each $P_{i,j}$ contributes a valley only if $P_{i,j}$ is non-empty and each $R_j$ is decomposed into $m$ such subpaths.\\
Part (c) corresponds to $Q_1,\ldots,Q_s$. More precisely, since each $Q_i$ has no return points and is non-empty we have the generating function $D^{(1)}(x,0;z)-1$.  Moreover, since the initial coordinate of each $Q_i$ is both a valley and a return point, we have to multiply the above by $xy$.

As before, we consider $F^{(t)}(x,y;z)=D^{(t)}(x,y;z)-1$.  Combining \eqref{eq:basic_functional_equation} and \eqref{eq:dyck_val_noret_gf_A} with \eqref{eq:full_functional_equation} yields
\begin{align}
	F^{(t)}(x,y;z) & = \sum_{s=0}^{\infty}(xy)^{s}z^{s+t}\Bigl(F^{(1)}(x,1;z)+1\Bigr)^{s+1}\Bigl(1+xF^{(1)}(x,1;z)\Bigr)^{m(s+t)-(s+1)}\label{eq:dyck_val_ret_gf_A}.
\end{align}
Using this notation, we may rephrase the definition of $\NPol_{m,n,t}(x,y)$ from \eqref{eq:mnt_valley_return_enumerator} as follows:
\begin{equation}\label{eq:mnt_narayana_gf}
	\NPol_{m,n,t}(x,y) = \langle z^{n}\rangle F^{(t)}(x,y;z).
\end{equation}

\begin{restatable}{theorem}{nform}\label{thm:mnt_narayana}
	For $m,n,t>0$, we have
	\begin{multline*}
		\NPol_{m,n,t}(x,y) = \sum_{a=0}^{n-t}\sum_{b=0}^{a}\Biggl(\binom{n-t}{a}\binom{mn-b-1}{a-b}\\
		- m\binom{n-t+1}{a+1}\binom{mn-b-2}{a-b-1}\Biggr)x^{a}y^{b}.
	\end{multline*}
\end{restatable}
\begin{proof}[Short proof]
	If we define $g(z)\defs (z+1)^{s+1}(xz+1)^{m(s+t)-(s+1)}$, then we may rephrase \eqref{eq:dyck_val_ret_gf_A} as
	\begin{equation}\label{eq:dyck_val_ret_gf_B}
		F^{(t)}(x,y;z) = \sum_{s=0}^{\infty} (xy)^{s}z^{s+t}g\bigl(F^{(1)}(x,1;z)\bigr).
	\end{equation}
	In view of \eqref{eq:mnt_narayana_gf} and \eqref{eq:dyck_val_ret_gf_B}, we apply Lemma~\ref{lem:lagrange_burmann} to the functions $f(z)$ and $g(z)$ to obtain the claimed formula.  Once again, the details are given in Appendix~\ref{app:mnt_narayana}.
\end{proof}

\begin{example}\label{ex:ntriangle_242}
	Let us illustrate the previous results on a small example.  For $n=4$ and $m=t=2$, the $25$ $(2,2$)-Dyck paths of height $4$ are shown in Figure~\ref{fig:242_paths}.  By ignoring the different colors, we notice that there are fourteen paths with two valleys, ten paths with one valley and one path without valleys.  This gives
	\begin{displaymath}
		\NPol_{2,4,2}(x,1) = 14x^{2} + 10x + 1
	\end{displaymath}
	in accordance with Proposition~\ref{prop:m_dyck_valley}.  If we count those paths without circled returns, then we get
	\begin{displaymath}
		\NPol_{2,4,2}(x,0) = 9x^{2} + 8x + 1
	\end{displaymath}
	in accordance with Proposition~\ref{prop:m_dyck_valley_noret}.  If we enumerate these paths with respect to valleys and returns, we get
	\begin{displaymath}
		\NPol_{2,4,2}(x,y) = x^{2}y^{2} + 4x^{2}y + 9x^{2} + 2xy + 8x + 1,
	\end{displaymath}
	which agrees with Theorem~\ref{thm:mnt_narayana}.
\end{example}

\begin{figure}
	\centering
	\includegraphics[page=4,width=\textwidth]{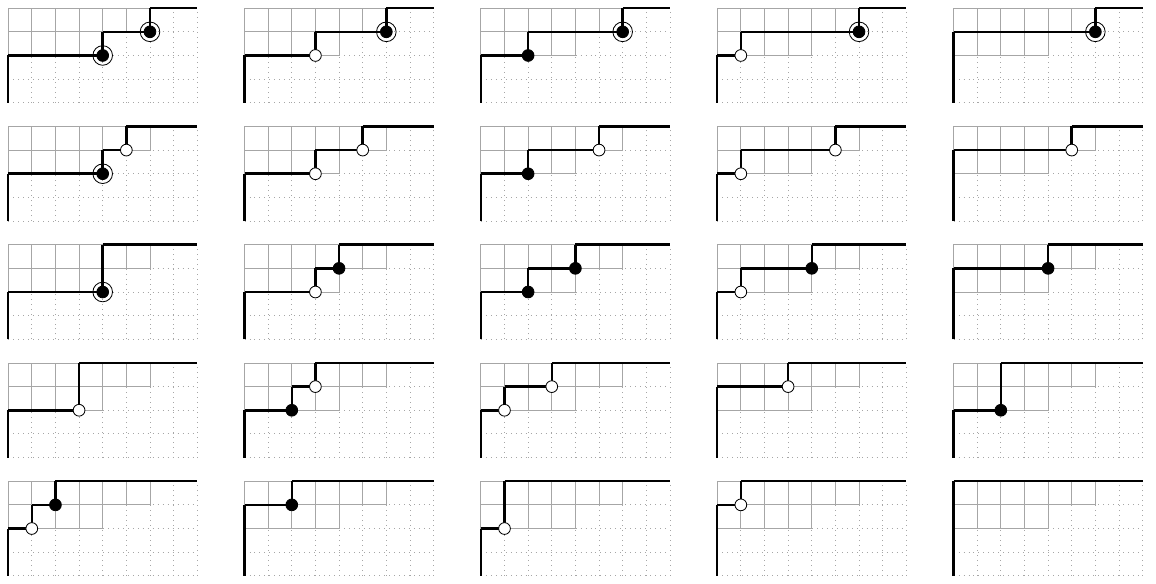}
	\caption{The $25$ $(2,2)$-Dyck paths of height $4$ with their valleys and returns highlighted.  The $2$-valleys are marked by solid dots, the $1$-valleys by hollow ones. Returns are circled.}
	\label{fig:242_paths}
\end{figure}

\subsection{$i$-valley enumeration for $(m,t)$-Dyck paths}
\label{sec:ivalley_enumeration}
In this section we want to enumerate $(m,t)$-Dyck paths with respect to $i$-valleys and returns.  Since any return is by definition an $m$-valley, we should expect a difference between enumerating paths with respect to $i$-valleys and returns, depending on $i$.  It will turn out, however, that there is only a difference between the cases $i=m$ and $i\neq m$.  
We therefore consider two polynomials:
\begin{align}
	\APol_{m,n,t}(x,y) & \defs \sum_{P\in\Dyck_{m,n,t}}x^{\val(P,m)}y^{\ret(P)}\label{eq:mnt_mvalley_return_enumerator},\\
	\BPol_{m,n,t}(x,y) & \defs \sum_{P\in\Dyck_{m,n,t}}x^{\val(P,i)}y^{\ret(P)}\label{eq:mnt_ivalley_return_enumerator}.
\end{align}
Our strategy is essentially the same as in Section~\ref{sec:valley_mnt}.

\subsubsection{The case $t=1$ without distinguishing the returns}
Let $P\in\Dyck_{m,\bullet,1}$.  For convenience, we repeat the decomposition from \eqref{eq:basic_decomposition}:
\begin{displaymath}
	P = NP_{0}EP_{1}E\cdots P_{m-1}EP_{m}.
\end{displaymath}
Now, let $i\in[m]$.  Then, an $i$-valley of $P$ may be an $i$-valley of $P_{0}$ or $P_{m}$, it may be an $i+j$-valley of $P_{j}$ or it may be the initial coordinate of $P_{i}$ whenever this path is not empty.  This gives the following functional equation for $C_{i}^{(1)}(x,1;z)$:
\begin{equation}\label{eq:mdyck_mval_gf_A}
	C_{i}^{(1)}(x,1;z) = 1 + z \prod_{\substack{j=0\\j\neq i}}^{m}C_{i-j}^{(1)}(x,1;z)\Bigl(1+x\Bigl(C_{m}^{(1)}(x,1;z)-1\Bigr)\Bigr).
\end{equation}
As before, the additional factor $z$ in the second summand accounts for the fact that $\lvert P_{0}\rvert+\cdots+\lvert P_{m}\rvert=\lvert P\rvert-1$.  There are $j$ of the separate east steps (\ie the steps in red in Figure \ref{fig:basic_decomposition}) appearing before the path $P_{j}$, which tells us that an $i$-valley of $P$ appearing within $P_{j}$ corresponds to an $i-j$-valley of $P_{j}$ regarded as a path starting from the origin.  Moreover, if $P_{j}$ is not empty, then its initial coordinate is an $j$-valley as well.  Therefore, since $i-i=0$, the path $P_{i}$, if not empty, contributes all its $m$-valleys as well as its initial coordinate to the $i$-valleys of $P$.  Subscripts in the above decomposition are always understood modulo $m$.

The following result justifies that it is enough to consider the polynomials $\APol_{m,n,t}(x,y)$ and $\BPol_{m,n,t}(x,y)$.

\begin{restatable}{proposition}{residue}\label{prop:residue_valleys}
	For $i,j\in[m]$, we have $C^{(1)}_{i}(x,1;z)=C^{(1)}_{j}(x,1;z)$. 
\end{restatable}
\begin{proof}
	This follows essentially from \cite[Theorem~1.1]{burstein20distribution}.  
\end{proof}

Let us define 
\begin{equation}\label{eq:idyck_enumerator_shifted}
	H_{i}^{(t)}(x,y;z) \defs C_{i}^{(t)}(x,y;z) - 1.
\end{equation}
Then Proposition~\ref{prop:residue_valleys} lets us rewrite \eqref{eq:mdyck_mval_gf_A} as
\begin{equation}\label{eq:mdyck_mval_gf_B}
	H_{i}^{(1)}(x,1;z) = z\Bigl(H_{i}^{(1)}(x,1;z)+1\Bigr)^{m}\Bigl(1+x\Bigl(H_{i}^{(1)}(x,1;z)\Bigr)\Bigr),
\end{equation}
which implies
\begin{equation}\label{eq:m_narayana_def}
	\APol_{m,n,1}(x,1) = \langle z^{n}\rangle H_{i}^{(1)}(x,1;z) = \BPol_{m,n,1}(x,1).
\end{equation}

\begin{restatable}{proposition}{aballform}\label{prop:m_dyck_mvalley}
	For $m,n>0$, we have 
	\begin{equation}
		\APol_{m,n,1}(x,1) = \BPol_{m,n,1}(x,1) = \sum_{r=0}^{n-1}\frac{1}{n}\binom{n}{r}\binom{mn}{n-r-1}x^{r}.
	\end{equation}
\end{restatable}
\begin{proof}[Short proof]
	By Proposition~\ref{prop:residue_valleys}, it follows that $C_{i}^{(1)}(x,1;z)$ does not depend on the choice of $i$.  Thus, $\APol_{m,n,1}(x,1)=\BPol_{m,n,1}(x,1)$.  
	
	Moreover, if we define $q_{\bullet}(z)\defs (z+1)^{m}(xz+1)$, then we may rewrite \eqref{eq:mdyck_mval_gf_B} as 
	\begin{equation}\label{eq:mdyck_mval_gf_C}
		H_{i}^{(1)}(x,1;z) = zq_{\bullet}\bigl(H_{i}^{(1)}(x,1;z)\bigr).
	\end{equation}
	Therefore the compositional inverse of $H_{i}^{(1)}(x,1;z)$ is
	\begin{equation}\label{eq:mdyck_mval_gf_D}
		h(z) \defs \frac{z}{(z+1)^{m}(xz+1)}.
	\end{equation}
	In Appendix~\ref{app:m_dyck_mvalley}, we apply Lemma~\ref{lem:lagrange_burmann} to the functions $h(z)$ and $q_{\bullet}(z)$ and obtain the result.
\end{proof}

\subsubsection{The case $t=1$ with no returns}

Analogously to the computation of the generating function for all valleys, we now consider the following generating function:
\begin{equation}\label{eq:mdyck_val_noret_gf_A}
	C^{(1)}_{i}(x,0;z) \defs \sum_{P\in\Dyck_{m,\bullet,1}^{+}}x^{\val(P,i)}z^{\lvert P\rvert}.
\end{equation}
As before, $P\in\Dyck_{m,\bullet,1}^{+}$ can be decomposed as in \eqref{eq:basic_decomposition} with $P_{m}=\epsilon$.  This means, however, that the initial coordinates of $P_{1},P_{2},\ldots,P_{m-1}$ never contribute an $m$-valley (regardless of them being empty or not), while exactly one of them contributes an additional $i$-valley for $i>0$ when $P_{i}\neq\epsilon$.

We therefore get the following functional equation once again appealing to Proposition~\ref{prop:residue_valleys}:
\begin{equation}\label{eq:mdyck_mval_noret_gf_B}
	C^{(1)}_{i}(x,0;z) = \begin{cases}1 + z\Bigl(C_{i}^{(1)}(x,1;z)\Bigr)^{m-1}\Bigl(1+x\Bigl(C_{i}^{(1)}(x;z)-1\Bigr)\Bigr), & \text{if}\;i<m,\\
		1 + z\Bigl(C_{m}^{(1)}(x,1;z)\Bigr)^{m}, & \text{if}\;i=m.\end{cases}
\end{equation}
Thus using the function $H_{i}^{(1)}(x,y;z)$ from \eqref{eq:idyck_enumerator_shifted}, we get:
\begin{align*}
	\APol_{m,n,1}(x,0) & = \langle z^{n}\rangle H_{m}^{(1)}(x,0;z),\\
	\BPol_{m,n,1}(x,0) & = \langle z^{n}\rangle H_{i}^{(1)}(x,0;z).
\end{align*}
We define
\begin{displaymath}
	q_{+,i}(z) \defs \begin{cases}(z+1)^{m-1}(xz+1), & \text{if}\;i<m,\\(z+1)^{m}, & \text{if}\;i=m\end{cases}
\end{displaymath}
so that \eqref{eq:mdyck_mval_noret_gf_B} becomes $H_{i}^{(1)}(x,0;z)=zq_{+,i}\bigl(H_{i}^{(1)}(x,1;z)\bigr)$.  We then apply Lemma~\ref{lem:lagrange_burmann} to $h(z)$ and $q_{+,i}(z)$ to obtain the following results; see Appendix~\ref{app:m_dyck_mvalley_noret} and \ref{app:m_dyck_mvalley_noret_var} for the details.

\begin{restatable}{proposition}{aposform}\label{prop:m_dyck_mvalley_noret}
	For $m,n>0$, we have
	\begin{displaymath}
		\APol_{m,n,1}(x,0) = \sum_{r=0}^{n-1}\frac{m}{n-1}\binom{n-1}{r}\binom{mn-1}{n-r-2}x^{r},\\
	\end{displaymath}
\end{restatable}

\begin{restatable}{proposition}{bposform}\label{prop:m_dyck_mvalley_noret_var}
	For $m,n>0$, we have
	\begin{displaymath}
		\BPol_{m,n,1}(x,0) = \sum_{r=0}^{n-1}\frac{1}{n}\binom{n}{r}\binom{mn-2}{n-r-1}x^{r}.
	\end{displaymath}
\end{restatable}

\subsubsection{The case $t>1$}

We are now ready to consider the general case, and finally prove Theorems~\ref{thm:mnt_htriangle} and \ref{thm:mnt_htriangle_var}.  
Our strategy is to apply steps similar  to those of the previous section for the generating function 
\begin{equation}
	C^{(t)}_{i}(x,0;z) \defs \sum_{P\in\Dyck_{m,\bullet,t}^{+}}x^{\val(P,i)}z^{\lvert P\rvert}.
\end{equation}

Let $P\in\Dyck_{m,\bullet,t}$ with $P\neq\epsilon$ and $\ret(P)=s$.  Recall from \eqref{eq:full_decomposition}, that we may decompose $P$ as
\begin{displaymath}
	P = N^{t}R_{1}R_{2}\cdots R_{t}Q_{1}Q_{2}\cdots Q_{s},
\end{displaymath}
where $R_{j}=P_{0,j}EP_{1,j}E\cdots P_{m-1,j}E$ with $P_{i,j}\in\Dyck_{m,\bullet,1}$ and $Q_{i}\in\Dyck_{m,\bullet,1}^{+}\setminus\{\epsilon\}$.

Let us first understand the distribution of the $m$-valleys.  For $j>1$, the first coordinate of $P_{0,j}\neq\epsilon$ contributes an $m$-valley, and each nonempty $P_{k,j}$ contributes its $m-k$-valleys.  Moreover, the $Q_{\ell}$'s contribute their $m$-valleys and one more $m$-valley, which is also a return, for their first coordinate.  Using Proposition~\ref{prop:residue_valleys} and \eqref{eq:mdyck_mval_noret_gf_B}, we conclude the following functional equation:
\begin{align}
	C_{m}^{(t)} & (x,y;z) = 1 + \sum_{s=0}^{\infty}z^{t}\Bigl(C_{m}^{(1)}(x,1;z)\Bigr)^{m}\notag\\
	& \kern2cm \times \Biggl(\Bigl(1+x\Bigl(C_{m}^{(1)}(x,1;z)-1\Bigr)\Bigr)\Bigl(C_{m}^{(1)}(x,1;z)\Bigr)^{m-1}\Biggr)^{t-1}\notag\\
	& \kern2.5cm \times (xy)^{s}\Bigl(C_{m}^{(1)}(x,0;z)-1\Bigr)^{s}\notag\\
	& = 1 + \sum_{s=0}^{\infty}(xy)^{s}z^{s+t}\notag\\
	& \kern1.5cm \times\Bigl(1+x\Bigl(C_{m}^{(1)}(x,1;z)-1\Bigr)\Bigr)^{t-1}\Bigl(C_{m}^{(1)}(x,1;z)\Bigr)^{(m-1)(t-1)+m(s+1)}.\label{eq:mnt_mval_ret_gf_A}
\end{align}
If we consider the distribution of the $i$-valleys for $i<m$ instead, then we see that the situation is slightly different.  The first coordinate of each $P_{i,j}\neq\epsilon$ contributes an $i$-valley and each nonempty $P_{k,j}$ contributes its $i-k$-valleys.  Moreover, the $Q_{\ell}$'s contribute their $i$-valleys, and their first coordinates contribute a return, but not an $i$-valley.  Using Proposition~\ref{prop:residue_valleys} and \eqref{eq:mdyck_mval_noret_gf_B}, we obtain the following functional equation:
\begin{align}
	C_{i}^{(t)} & (x,y;z) = 1 + \sum_{s=0}^{\infty}z^{t}\Biggl(\Bigl(1+x\Bigl(C_{i}^{(1)}(x,1;z)-1\Bigr)\Bigr)\Bigl(C_{i}^{(1)}(x,1;z)\Bigr)^{m-1}\Biggr)^{t}\notag\\
	& \kern3cm \times y^{s}\Bigl(C_{i}^{(1)}(x,0;z)-1\Bigr)^{s}\notag\\
	& = 1 + \sum_{s=0}^{\infty}y^{s}z^{s+t}\Bigl(1+x\Bigl(C_{i}^{(1)}(x,1;z)-1\Bigr)\Bigr)^{s+t}\Bigl(C_{i}^{(1)}(x,1;z)\Bigr)^{(s+t)(m-1)}.\label{eq:mnt_mval_ret_gf_var_A}
\end{align}
In fact, Proposition~\ref{prop:residue_valleys} implies that this functional equation does not depend on $i$.  We define
\begin{equation}
	q_{i}(z) \defs \begin{cases}
		(xz+1)^{s+t}(z+1)^{(s+t)(m-1)}, & \text{if}\;i<m,\\
		(xz+1)^{t-1}(z+1)^{(m-1)(t-1)+m(s+1)}, & \text{if}\;i=m.
	\end{cases}
\end{equation}
Then, using $H_{i}^{(t)}(x,y;z)$ from \eqref{eq:idyck_enumerator_shifted}, we get
\begin{equation}\label{eq:mdyck_mval_ret_gf_A}
	H_{i}^{(t)}(x,y;z) = \begin{cases}
		\sum\limits_{s=0}^{\infty}y^{s}z^{s+t}q_{i}\Bigl(H_{i}^{(1)}(x,1;z)\Bigr), & \text{if}\;i<m,\\
		\sum\limits_{s=0}^{\infty}(xy)^{s}z^{s+t}q_{i}\Bigl(H_{i}^{(1)}(x,1;z)\Bigr), & \text{if}\;i=m,
	\end{cases}
\end{equation}
and we obtain
\begin{align*}
	\APol_{m,n,t}(x,y) & = \langle z^{n}\rangle H_{m}^{(t)}(x,y;z),\\
	\BPol_{m,n,t}(x,y) & = \langle z^{n}\rangle H_{i}^{(t)}(x,y;z).
\end{align*}

\begin{restatable}{theorem}{aform}\label{thm:mnt_htriangle}
	For $m,n,t>0$, we have
	\begin{multline*}
		\APol_{m,n,t}(x,y) = \sum_{a=0}^{n-t}\sum_{b=0}^{a}\Biggl(\binom{n-b-2}{a-b}\binom{mn-t+1}{n-t-a}\\
		- m\binom{n-b-1}{a-b}\binom{mn-t}{n-t-a-1}\Biggr)x^{a}y^{b}.
	\end{multline*}
\end{restatable}

\begin{restatable}{theorem}{bform}\label{thm:mnt_htriangle_var}
	For $m,n,t>0$, we have
	\begin{displaymath}
		\BPol_{m,n,t}(x,y) = \sum_{a=0}^{n-t}\sum_{b=0}^{n-t-a}\frac{t+b}{n}\binom{n}{a}\binom{mn-t-b-1}{n-t-a-b}x^{a}y^{b}.
	\end{displaymath}
\end{restatable}

The proof of Theorem~\ref{thm:mnt_htriangle} can be found in Appendix~\ref{app:mnt_htriangle} and the proof of Theorem~\ref{thm:mnt_htriangle_var} in Appendix~\ref{app:mnt_htriangle_var}.  

\begin{example}\label{ex:atriangle_242}
	Let us return to the example from Figure~\ref{fig:242_paths}, and consider $m=t=2$ and $n=4$.  If $i=1$, then count the paths with respect to the valleys indicated by a hollow node.  Without distinguishing the returns, there are three paths with exactly two $1$-valleys, $13$ paths with exactly one $1$-valley and nine paths without $1$-valleys.  The same distribution is obtained when considering $2$-valleys.  We thus get
	\begin{displaymath}
		\APol_{2,4,2}(x,1) = \BPol_{2,4,2}(x,1) = 3x^{2} + 13x + 9,
	\end{displaymath}
	in accordance with Proposition~\ref{prop:m_dyck_mvalley}.  Counting the paths without returns with respect to their $2$- or $1$-valleys, respectively, yields
	\begin{align*}
		\APol_{2,4,2}(x,0) & = x^{2} + 8x + 9,\\
		\BPol_{2,4,2}(x,0) & = 3x^{2} + 10x + 5,
	\end{align*}
	which agrees with Propositions~\ref{prop:m_dyck_mvalley_noret} and \ref{prop:m_dyck_mvalley_noret_var}, respectively.  For the refined enumeration with respect to valleys and returns, we get
	\begin{align*}
		\APol_{2,4,2}(x,y) & = x^{2}y^{2} + x^{2}y + x^{2} + 5xy + 8x + 9,\\
		\BPol_{2,4,2}(x,y) & = 3x^{2} + 3xy + y^{2} + 10x + 3y + 5,
	\end{align*}
	in accordance with Theorems~\ref{thm:mnt_htriangle} and \ref{thm:mnt_htriangle_var}.
\end{example}

\begin{remark}\label{rem:atriangle_origin}
	As mentioned in the introduction, the initial motivation for this article was to give a combinatorial interpretation of the polynomial $\tilde{\mathcal{H}}_{m,n,t}(x,y)$ from~\eqref{eq:htriangle}, which is computed in the proof of \cite[Theorem~4.3]{krattenthaler22rank} and is defined through a variable substitution from a certain weighted rank-generating polynomial associated with the poset of generalized $m$-divisible noncrossing $t$-partitions studied in \cite{krattenthaler22rank}.  Theorem~\ref{thm:mnt_htriangle} asserts that this polynomial is exactly $\APol_{m,n,t}(x,y)$, which provides the desired combinatorial explanation.

	This interpretation matches nicely with Remark~\ref{rem:reverse_fuss_narayana}, because if we plug in $y=1$, then we obtain
	\begin{align*}
		\APol_{m,n,t}(x,1) 
		& = \sum_{a=0}^{n-t}\frac{mnt-(t-1)(n-a)}{n(mn-t+1)}\binom{mn-t+1}{n-t-a}\binom{n}{a}x^{a},
	\end{align*}
	which is the rank-generating polynomial of the poset of $m$-divisible noncrossing $t$-partitions; \cite[Theorem~3.1]{krattenthaler22rank}.
\end{remark}

We have thus finished the proof of Theorem~\ref{thm:valley_return_enumeration}.

\begin{proof}[Proof of Theorem~\ref{thm:valley_return_enumeration}]
	This follows from Theorems~\ref{thm:mnt_narayana}, \ref{thm:mnt_htriangle} and \ref{thm:mnt_htriangle_var}.
\end{proof}

\subsection{Interactions between $\APol_{m,n,1}(x,y)$, $\BPol_{m,n,1}(x,y)$ and $\NPol_{m,n,1}(x,y)$}
	\label{sec:poly_relations}

We have mentioned it in Remark~\ref{rem:fuss_narayana} and we may notice by comparing Propositions~\ref{prop:m_dyck_valley} and \ref{prop:m_dyck_mvalley} that the coefficient sequences of $\NPol_{m,n,1}(x,1)$ and $\BPol_{m,n,1}(x,1)$ are reverse to each other.  This connection has been observed before in \cite[Corollary~1.12]{burstein20distribution}.  In terms of polynomials, this relation means that we can obtain $\BPol_{m,n,1}(x,1)$ by first replacing $x$ by $\frac{1}{x}$ in $\NPol_{m,n,1}(x,1)$ and then multiplying the result by $x^{n-1}$.  Using the explicit formulas for the coefficients, we may extend this connection to all values of $y$.
	
\begin{lemma}\label{lem:nb_relation}
	For $n>0$ and $m\in\mathbb{Z}$, we get
	\begin{displaymath}
		\BPol_{m,n,1}(x,y) = x^{n-1}\NPol_{m,n,1}\left(\frac{1}{x},y\right).
	\end{displaymath}
\end{lemma}
\begin{proof}
	This is a straightforward computation using the explict formulas from Theorems~\ref{thm:mnt_narayana} and \ref{thm:mnt_htriangle_var}.
\end{proof}
	
It is straightforward to verify that there exists a similar connection between $\APol_{m,n,1}(x,y)$ and $\NPol_{m,n,1}(x,y)$, and consequently also between $\APol_{m,n,1}(x,y)$ and $\BPol_{m,n,1}(x,y)$.

\begin{lemma}\label{lem:na_relation}
	For $n>0$ and $m\in\mathbb{Z}$, we get
	\begin{align*}
		\APol_{m,n,1}(x,y) = x^{n-1}\NPol_{m,n,1}\left(\frac{1}{x},x(y-1)+1\right).
	\end{align*}
\end{lemma}
\begin{proof}
	This is a straightforward computation using the explicit formulas from Theorems~\ref{thm:mnt_narayana} and \ref{thm:mnt_htriangle}.
\end{proof}

And finally, combining  the above two lemmas we get the following. 
\begin{lemma}\label{lem:ab_relation}
	For $n>0$ and $m\in\mathbb{Z}$, we get
	\begin{displaymath}
		\APol_{m,n,1}(x,y) = \BPol_{m,n,1}\bigl(x,x(y-1)+1\bigr).
	\end{displaymath}
\end{lemma}

\begin{remark}
	For $t>1$, the relations in this section do no longer hold.  Consider $m=t=2$ and $n=4$.  We get
	\begin{align*}
		\APol_{2,4,2}(x,y) & = x^{2}y^{2} + x^{2}y + x^{2} + 5xy + 8x + 9,\\
		\BPol_{2,4,2}(x,y) & = 3x^{2} + 3xy + 10x + 5,\\
		\NPol_{2,4,2}(x,y) & = x^{2}y^{2} + 4x^{2}y + 9x^{2} + 2xy + 8x + 1.
	\end{align*}
	Then,
	\begin{align*}
		x^{2}\NPol_{2,4,2}\left(\frac{1}{x},y\right) & = x^{2}+2xy+y^{2}+8x+4y+9 \neq \BPol_{2,4,2}(x,y),\\
		x^{2}\NPol_{2,4,2}\left(\frac{1}{x},x(y-1)+1\right) & = x^{2}y^{2}+6xy+4x+14 \neq \APol_{2,4,2}(x,y),\\
		\BPol_{2,4,2}\bigl(x,x(y-1)+1\bigr) & = 3x^{2}y+13x+5 \neq \APol_{2,4,2}(x,y).
	\end{align*}
	
	As of now, it is not clear to us what the relations from Lemmas~\ref{lem:nb_relation}--\ref{lem:ab_relation} mean combinatorially.
\end{remark}

\section{Enumeration of certain Schr{\"o}der paths with respect to their number of returns and diagonal steps}
\label{sec:schroder_paths}

\subsection{Northeast paths with possible diagonal steps}

We now consider lattice paths drawing from a slightly larger step set.  A \defn{Delannoy path} is a lattice path starting at the origin and using steps of type $N=(0,1)$ (north steps), $E=(1,0)$ (east steps), but also possibly steps of type $D\defs(1,1)$ (\defn{diagonal steps}).
An \defn{$(m,t)$-Schr{\"o}der path} of height $n$ is a Delannoy path from $(0,0)$ to $(mn,n)$ which starts with at least $t$ north steps and never goes below the line $x=my$.  

An $(m,t)$-Schr{\"o}der path is \defn{small} if it has no diagonal step starting on the line $x=my$, and it is \defn{positive} if it never touches the line $x=my$, except in $(0,0)$ and $(mn,n)$.  We write $\Schroder_{m,n,t}$ for the set of small $(m,t)$-Schr{\"o}der paths of height $n$ and $\Schroder_{m,n,t}^{+}$ for the set of positive $(m,t)$-Schr{\"o}der paths.  See Figure~\ref{fig:delannoy} for an illustration of these definitions.

\begin{figure}[h]
	\centering
	\includegraphics[width=\textwidth,page=19]{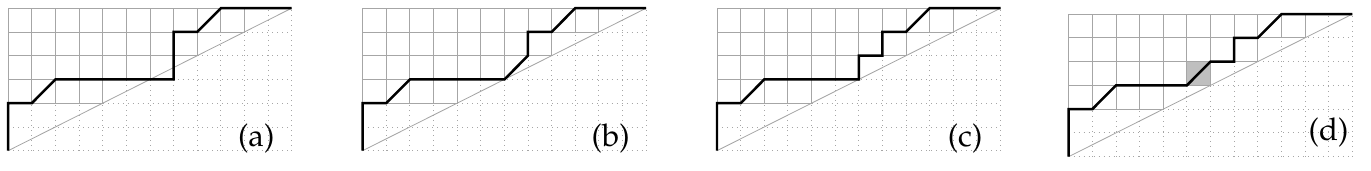}
	\caption{Several northeast paths with diagonal steps.  (a) A Delannoy path, (b) a $(2,2)$-Schr{\"o}der path of height $6$, (c) a small $(2,2)$-Schr{\"o}der path of height $6$, (d) a positive $(2,2)$-Schr{\"o}der path of height $6$.  The highlighted diagonal in (d) is cornered.}
	\label{fig:delannoy}
\end{figure}

\begin{remark}
	By definition, any small $(m,t)$-Schr{\"o}der path of height $n$ can---after removing the first $t$ north steps and the last $m$ east steps---be drawn inside a Ferrers shape of type $\bigl(mt, m(t+1),\ldots,m(n-1)\bigr)$.
\end{remark}

Let $S$ be an $(m,t)$-Schr{\"o}der path of height $n$.  The $i$-valleys and returns of $S$ are defined analogously to the case of $(m,t)$-Dyck paths.  We write $\dgn(S)$ for the number of diagonal steps of $S$.  
The \defn{codimension} of $S$ is $\codim(S)\defs n-t-\dgn(S)$.

\begin{remark}
	The notion of \emph{codimension} comes from the special case $t=1$.  It was shown in \cite[Corollary~4.8]{bell21schroder} that the $(m,1)$-Schr{\"o}der paths of height $n$ index the faces of a certain polytopal complex.  Then, the dimension of the face associated with a Schr{\"o}der path $S$ is given by $\dgn(S)$. 
\end{remark}

An \defn{$m$-diagonal} of $S$ is a diagonal step ending on an $x$-coordinate divisible by $m$.  We denote by $\dgn(S,m)$ the number of $m$-diagonals of $S$.  A \defn{cornered diagonal} is a diagonal step connecting $(im-1,i)$ with $(im,i+1)$ for some $i\in[n]$.  Then, any cornered diagonal is an $m$-diagonal, and we write $\cd(S)$ for the number of cornered diagonals of $S$.  

An $(m,t)$-Schr{\"o}der path is \defn{$m$-divisible} if all of its diagonal steps are $m$-diagonals.  We write $\Schroder_{m,n,t}^{(m)}$ for the set of $m$-divisible small $(m,t)$-Schr{\"o}der paths of height $n$. 

In this section, we want to enumerate certain $(m,t)$-Schr{\"o}der paths with respect to some of the parameters we have just introduced. 
Although it is possible to compute the relevant formulas through generating functions, we want to emphasize a combinatorial method to transform $(m,t)$-Dyck paths into small $(m,t)$-Schr{\"o}der paths, and record this transformation in the counting polynomials $\NPol_{m,n,t}(x,y)$ and $\APol_{m,n,t}(x,y)$ as certain variable substitutions.  Then, we compute the explicit formulas using Theorem~\ref{thm:valley_return_enumeration}.   This transformation essentially passes from $(m,t)$-Dyck paths to small $(m,t)$-Schr{\"o}der paths by replacing valleys by diagonal steps; see also \cite[Remark~3.2]{ceballos21fh} and \cite[Theorem~4.7]{bell21schroder}.  It is illustrated in Figure~\ref{fig:dyck_to_schroder}.

To make this precise, we introduce the following sets:
\begin{align}
	\mathfrak{F}_{m,n,t} & \defs \bigl\{(P,A)\colon P\in\Dyck_{m,n,t},A\;\text{is a subset of the valleys of}\;P\bigr\},\\
	\mathfrak{F}_{m,n,t}^{(m)} & \defs \bigl\{(P,A)\colon P\in\Dyck_{m,n,t},A\;\text{is a subset of the $m$-valleys of}\;P\bigr\}.
\end{align}

\begin{lemma}\label{lem:asso_schroder_bijection}
	Fix integers $m,n,t>0$.  There is a bijection from $\mathfrak{F}_{m,n,t}$ to $\Schroder_{m,n,t}$ with the property that when $(P,A)\mapsto S$, then
	\begin{align*}
		\val(P) & = \val(S)+\dgn(S),\\
		\ret(P) & = \ret(S)+\cd(S).
	\end{align*}
	This map restricts to a bijection from $\mathfrak{F}_{m,n,t}^{(m)}$ to $\Schroder_{m,n,t}^{(m)}$.
\end{lemma}
\begin{proof}
	Let $P\in\Dyck_{m,n,t}$.  If $P$ has a valley at a coordinate $(p,q)$, then the path necessarily contains the coordinates $(p-1,q)$ and $(p,q+1)$, too.  We may thus replace such a valley by a diagonal step connecting $(p-1,q)$ and $(p,q+1)$.  Since $P\in\Dyck_{m,n,t}$ it is follows that $p\leq mq$, and the resulting diagonal step therefore does not start on the line $x=my$.  This implies that every $(P,A)\in\mathfrak{F}_{m,n,t}$ gets mapped to an element of $\Schroder_{m,n,t}$.

	Conversely, suppose that $S\in\Schroder_{m,n,t}$ has a diagonal step connecting $(p,q)$ and $(p+1,q+1)$.  Since this diagonal step cannot start on the line $x=my$, it must hold that $p<mq$.  But then, converting this diagonal step into a step pair $EN$ yields a path with a valley at $(p+1,q)$.  This path stays weakly above the line $x=my$.  After converting all diagonal steps of $S$ we thus obtain some $P\in\Dyck_{m,n,t}$.  If $A$ is the set of newly constructed valleys, then $(P,A)\in\mathfrak{F}_{m,n,t}$.

	It is now clear that this map is a bijection, and the claim about the statistics follows right away.
\end{proof}

\begin{figure}
	\centering
	\includegraphics[width=\textwidth,page=20]{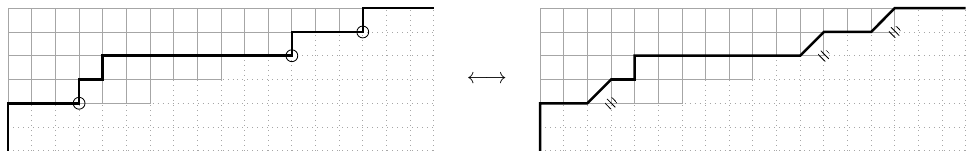}
	\caption{(left) A $(3,2)$-Dyck path of height $6$ with three of its four valleys marked.  (right) The corresponding small $(3,2)$-Schr{\"o}der path of height $6$ under the map from Lemma~\ref{lem:asso_schroder_bijection}.}
	\label{fig:dyck_to_schroder}
\end{figure}

\subsection{Codimension-return enumeration for small $(m,t)$-Schr{\"o}der paths}

We start by enumerating small $(m,t)$-Schr{\"o}der paths with respect to their codimension and their number of returns.  We thus consider the following counting polynomials:
\begin{align}
	\FPol_{m,n,t}(x,y) & \defs \sum_{S\in\Schroder_{m,n,t}}x^{\codim(S)-\ret(S)}y^{\ret(S)},\\
	\FPol_{m,n,t}^{(m)}(x,y) & \defs \sum_{S\in\Schroder_{m,n,t}^{(m)}}x^{\codim(S)-\ret(S)}y^{\ret(S)}\label{eq:mnt_ftriangle}.
\end{align}
We wish to emphasize that the only diagonal steps used in the paths in $\Schroder_{m,n,t}^{(m)}$ are $m$-diagonals.  Thus, $\codim$ is indeed the correct statistic in \eqref{eq:mnt_ftriangle}.

\begin{proposition}\label{prop:m_dyck_ftriangle_all}
	For $m,n,t>0$ we have
	\begin{displaymath}
		\FPol_{m,n,t}(x,y) = x^{n-t}\NPol_{m,n,t}\left(\frac{x+1}{x},\frac{y+1}{x+1}\right).
	\end{displaymath}
\end{proposition}
\begin{proof}
	Let $P\in\Dyck_{m,n,t}$ and let $V=\valset(P)$ denote the set of valleys of $P$.  Then, clearly, $V$ can be partitioned as $V=V_{1}\uplus V_{2}$ where $V_{2}=\retset(P)$ is the set of returns of $P$ and $V_{1}=V\setminus V_{2}$ is the set of non-return valleys of $P$.  Likewise, if we pick $(P,A)\in\mathfrak{F}_{m,n,t}$, then we can partition it as $A=A_{1}\uplus A_{2}$, where $A_{2}=\retset(P)\cap A$ and $A_{1}=A\setminus A_{2}$.

	We start by expanding the right-hand side:
	\begin{align*}
		x^{n-t} & \NPol_{m,n,t}\left(\frac{x+1}{x},\frac{y+1}{x+1}\right)\\
		& = \sum_{P\in\Dyck_{m,n,t}}x^{n-t-\val(P)}(x+1)^{\val(P)-\ret(P)}(y+1)^{\ret(P)}\\
		& = \sum_{\substack{P\in\Dyck_{m,n,t}\\\valset(P)=V_{1}\uplus V_{2}}}x^{n-t-\val(P)\rvert}\sum_{k=0}^{\lvert V_{1}\rvert}\binom{\lvert V_{1}\rvert}{k}x^{k}\sum_{\ell=0}^{\lvert V_{2}\rvert}\binom{\lvert V_{2}\rvert}{\ell}y^{\ell}\\
		& \overset{\rm(a)}{=} \sum_{P\in\Dyck_{m,n,t}}x^{n-t-\val(P)}\sum_{A_{2}\subseteq\valset(P)\cap\retset(P)}\sum_{A_{1}\subseteq\valset(P)\setminus A_{2}}x^{\lvert A_{1}\rvert}y^{\lvert A_{2}\rvert}\\
		& = \sum_{P\in\Dyck_{m,n,t}}\sum_{A_{2}\subseteq\valset(P)\cap\retset(P)}\sum_{A_{1}\subseteq\valset(P)\setminus A_{2}}x^{n-t-\val(P)+\lvert A_{1}\rvert}y^{\lvert A_{2}\rvert}\\
		& \overset{\rm(b)}{=} \sum_{\substack{(P,A)\in\mathfrak{F}_{m,n,t}\\A=A_{1}\uplus A_{2}}}x^{n-t-\val(P)+\lvert A_{1}\rvert}y^{\lvert A_{2}\rvert}\\
		& \overset{\rm(c)}{=} \sum_{S\in\Schroder_{m,n,t}}x^{n-t-\dgn(S)-\ret(S)}y^{\ret(S)}\\
		& = \sum_{S\in\Schroder_{m,n,t}}x^{\codim(S)-\ret(S)}y^{\ret(S)}\\
		& = \FPol_{m,n,t}(x,y).
	\end{align*}
	Let us briefly explain what happens in the above equations.
	\begin{enumerate}[\rm (a)]
		\item Resolving the binomial coefficients by directly summing over all subsets of the valley set of $P$ essentially marks a set $A_{1}$ of non-return valleys and a set $A_{2}$ of returns.
		\item We rewrite the previous equation in terms of $\mathfrak{F}_{m,n,t}$.
		\item Here, we apply the bijection from Lemma~\ref{lem:asso_schroder_bijection}, where we convert all returns that are \emph{not} selected by $A_{2}$ to (cornered) diagonals, and all non-return valleys that are \emph{not} selected by $A_{1}$ to diagonals.  This means that the returns of the resulting $(m,t)$-Schr{\"o}der path are the ones selected by $A_{2}$: $\ret(S)=\lvert A_{2}\rvert$.  The non-return valleys are those selected by $A_{1}$.  For the total number of diagonals, we thus obtain $\dgn(S)=\val(P)-\val(S)=\val(P)-\lvert A_{1}\rvert-\lvert A_{2}\rvert = \val(P)-\lvert A_{1}\rvert-\ret(S)$.
	\end{enumerate}
\end{proof}

\begin{corollary}\label{cor:dyck_ftriangle_form}
	For $m,n,t>0$, we have
	\begin{displaymath}
		\FPol_{m,n,t}(x,y) = \sum_{a=0}^{n-t}\sum_{b=0}^{n-t-a}\frac{t+b}{n}\binom{mn+a-1}{a}\binom{mn}{mn-n+t+a+b}x^{a}y^{b}.
	\end{displaymath}
\end{corollary}
\begin{proof}
	This follows from Theorem~\ref{thm:mnt_narayana} and Proposition~\ref{prop:m_dyck_ftriangle_all} by direct computations.
\end{proof}

\begin{proposition}\label{prop:m_dyck_ftriangle}
	For $m,n,t>0$ we have
	\begin{displaymath}
		\FPol_{m,n,t}^{(m)}(x,y) = x^{n-t}\APol_{m,n,t}\left(\frac{x+1}{x},\frac{y+1}{x+1}\right).
	\end{displaymath}
\end{proposition}
\begin{proof}
	This is essentially verbatim to the proof of Proposition~\ref{prop:m_dyck_ftriangle_all}, but this time we restrict ourselves to $m$-valleys only.  This is valid since every return is by definition an $m$-valley.
\end{proof}

\begin{corollary}\label{cor:m_dyck_ftriangle_form}
	For $m,n,t>0$, we have
	\begin{displaymath}
		\FPol_{m,n,t}^{(m)}(x,y) = \sum_{a=0}^{n-t}\sum_{b=0}^{n-t-a}\frac{t+b}{n}\binom{mn+a-1}{a}\binom{n}{t+a+b}x^{a}y^{b}.
	\end{displaymath}
\end{corollary}
\begin{proof}
	This follows from Theorem~\ref{thm:mnt_htriangle} and Proposition~\ref{prop:m_dyck_ftriangle} by direct calculations.
\end{proof}

\begin{proof}[Proof of Theorem~\ref{thm:dimension_return_enumeration}]
	This follows from Corollaries~\ref{cor:dyck_ftriangle_form} and \ref{cor:m_dyck_ftriangle_form}.
\end{proof}

\begin{remark}\label{rem:parabolic_ftriangle}
	In Remark~\ref{rem:atriangle_origin}, we have explained that one of the main motivations for this article was to give a combinatorial explanation of $\APol_{m,n,t}(x,y)$, which arose in \cite{krattenthaler22rank} as a transformation of a certain weighted rank-generating polynomial associated with the poset of $m$-divisible noncrossing $t$-partitions.  The polynomial $\FPol_{m,n,t}^{(m)}(x,y)$ arises analogously via the transformation stated in Proposition~\ref{prop:m_dyck_ftriangle} and its coefficients were explicitly computed in \cite[Theorem~4.3]{krattenthaler22rank}.  We have just given a combinatorial realization of these coefficients in terms of small $(m,t)$-Schr{\"o}der paths.
\end{remark}

\begin{remark}\label{rem:parabolic_associahedron}
	For $t=1$, the coefficients of $\FPol_{m,n,1}^{(m)}(x,y)$ agree with the refined face numbers of $m$-generalized cluster complex in type $A_{n-1}$ that were previously computed in \cite[Theorem~FA]{krattenthaler06ftriangle}.

	Moreover, this connection suggests that the elements of $\mathfrak{F}_{m,n,1}^{(m)}$ correspond to the faces of the $m$-generalized cluster complex in type $A_{n-1}$.  Does the set $\mathfrak{F}_{m,n,t}^{(m)}$ possess an analogous cell complex structure when $t>1$?  Is there a similar structure realizing $\mathfrak{F}_{m,n,t}$?
\end{remark}

\begin{remark}
	An $(m,t)$-Dyck path of height $n$ can be viewed as a northeast path that starts with at least $t$ north steps and stays weakly above the path $\nu=(NE^{m})^{n}$.  If we use an arbitrary northeast path $\nu$ instead, then---for $t=0$---we recover the $\nu$-paths studied for instance in \cite{ceballos20nu}.  The analogues of Schr{\"o}der paths were considered for instance in \cite{bell21schroder}.  If we adapt the definition of \emph{return} to being a valley of $\nu$ and of \emph{cornered diagonal} to being a diagonal that isolates a return of $\nu$, then Proposition~\ref{prop:m_dyck_ftriangle_all} carries over to this setting.
\end{remark}

\begin{remark}
	In this section, we have only considered the polynomials $\NPol_{m,n,t}(x,y)$ and $\APol_{m,n,t}(x,y)$.  In principle, we could attempt the same transformations for the polynomial $\BPol_{m,n,t}(x,y)$.  However, since returns are necessarily $m$-valleys, this polynomial involves two types of valleys with different remainder modulo $m$.  Therefore, transforming some of the occurring valleys into diagonal steps would produce an $(m,t)$-Schr{\"o}der path with two different types of diagonal steps, which does not appear natural to us.
\end{remark}

\subsection{Diagonal-step enumeration for small $(m,t)$-Schr{\"o}der paths}

In this section, we enumerate small $(m,t)$-Schr{\"o}der paths with respect to their number of diagonals and cornered diagonals.  Even though this enumeration seems more natural, we have explained in Remarks~\ref{rem:parabolic_ftriangle} and~\ref{rem:parabolic_associahedron} why we have started with the polynomials $\FPol_{m,n,t}(x,y)$ and $\FPol_{m,n,t}^{(m)}(x,y)$.  We now consider the following polynomials:
\begin{align}
	\SPol_{m,n,t}(x,y) & \defs \sum_{S\in\Schroder_{m,n,t}}x^{\dgn(S)}y^{\cd(S)},\\
	\SPol_{m,n,t}^{(m)}(x,y) & \defs \sum_{S\in\Schroder_{m,n,t}^{(m)}}x^{\dgn(S,m)}y^{\cd(S)}.
\end{align}

\begin{proposition}\label{prop:schroder_diags}
	For $m,n,t>0$, we have
	\begin{displaymath}
		\SPol_{m,n,t}(x,y) = \NPol_{m,n,t}\left(x+1,\frac{xy+1}{x+1}\right).
	\end{displaymath}
\end{proposition}
\begin{proof}
	Let us, once again, partition a subset $A\subseteq\valset(P)$ of the valleys of $P\in\Dyck_{m,n,t}$ as $A=A_{1}\uplus A_{2}$, where $A_{2}=A\cap\retset(P)$ and $A_{1}=A\setminus A_{2}$.  Then, we get
	\begin{align*}
		\NPol_{m,n,t}\left(x+1,\frac{xy+1}{x+1}\right) & = \sum_{P\in\Dyck_{m,n,t}}(x+1)^{\val(P)-\ret(P)}(xy+1)^{\ret(P)}\\
		& = \sum_{\substack{(P,A)\in\mathfrak{F}_{m,n,t}\\A=A_{1}\uplus A_{2}}}x^{\lvert A_{1}\rvert}(xy)^{\lvert A_{2}\rvert}\\
		& = \sum_{\substack{(P,A)\in\mathfrak{F}_{m,n,t}\\A=A_{1}\uplus A_{2}}}x^{\lvert A\rvert}y^{\lvert A_{2}\rvert}\\
		& = \sum_{S\in\Schroder_{m,n,t}}x^{\dgn(S)}y^{\cd(S)}\\
		& = \SPol_{m,n,t}(x,y).\qedhere
	\end{align*}
\end{proof}

\begin{corollary}\label{cor:schroder_diags_form}
	For $m,n,t>0$, we have
	\begin{multline*}
		\SPol_{m,n,t}(x,y) = \sum_{a=0}^{n-t}\sum_{b=0}^{a}\Biggl(\sum_{i=1}^{n-t+1}\Biggl(\frac{(b+1)(m-1)}{mn-i+a-b} + \frac{(n-i)(mn-i+a+1)}{n(mn-i+a-b)} - \frac{n-t-i+1}{n-t+1}\Biggr)\\
			\times\binom{mn-i+a-b}{a-b}\binom{n-t+1}{n-t-i+1}\binom{mn}{i-a-1}\Biggr)x^{a}y^{b}.
	\end{multline*}
\end{corollary}
\begin{proof}
	This is a direct computation using Theorem~\ref{thm:mnt_narayana} and Proposition~\ref{prop:schroder_diags}.
\end{proof}

\begin{proposition}\label{prop:schroder_mdiags}
	For $m,n,t>0$, we have
	\begin{displaymath}
		\SPol_{m,n,t}^{(m)}(x,y) = \APol_{m,n,t}\left(x+1,\frac{xy+1}{x+1}\right).
	\end{displaymath}
\end{proposition}
\begin{proof}
	As before, this is essentially verbatim to the proof of Proposition~\ref{prop:schroder_diags}, but restricting to $m$-valleys and $m$-diagonals.
\end{proof}

\begin{corollary}\label{cor:schroder_mdiags_form}
	For $m,n,t>0$, we have
	\begin{multline*}
		\SPol_{m,n,t}^{(m)}(x,y) = \sum_{a=0}^{n-t}\sum_{b=0}^{a}\Biggl(\sum_{i=1}^{n-t+1}\Biggl(\frac{(n-i)(n-i+a+1)}{n(n-i+a-b)}-\frac{m(n-t-i+1)}{mn-t+1}\Biggr)\\
			\times\binom{n-i+a-b}{a-b}\binom{mn-t+1}{n-t-i+1}\binom{n}{i-a-1}\Biggr)x^{a}y^{b}.
	\end{multline*}
\end{corollary}
\begin{proof}
	This is a direct calculation using Theorem~\ref{thm:mnt_htriangle} and Proposition~\ref{prop:schroder_mdiags}.
\end{proof}

\begin{proof}[Proof of Theorem~\ref{thm:diagonal_enumeration}]
	This follows from Corollaries~\ref{cor:schroder_diags_form} and \ref{cor:schroder_mdiags_form}.
\end{proof}

\begin{remark}
	When $m=t=1$, then the evaluation of $\SPol_{1,n,1}^{(1)}(x,y)=\SPol_{1,n,1}(x,y)$ at $x=y=1$ yields the small Schr{\"o}der numbers, given by
	\begin{displaymath}
		\sum_{i=1}^{n}\frac{1}{n}\binom{n}{i}\binom{n}{i-1}2^{i-1}.
	\end{displaymath}
	This connection is the motivation for naming the elements of $\Schroder_{m,n,t}$ \emph{small $(m,t)$-Schr{\"o}der paths}.
\end{remark}

\subsection{Diagonal-step enumeration for positive $(m,t)$-Schr{\"o}der paths}

Our approach can also be used to enumerate positive $(m,t)$-Schr{\"o}der paths, \ie small $(m,t)$-Schr{\"o}der paths that touch the line $x=my$ only at the beginning and at the end.  We define the following polynomials:
\begin{align}
	\PPol_{m,n,t}(x,y) & \defs \sum_{S\in\Schroder_{m,n,t}^{+}}x^{\dgn(S)}y^{\cd(S)},\\
	\PPol_{m,n,t}^{(m)}(x,y) & \defs \sum_{S\in\Schroder_{m,n,t}^{(m)+}}x^{\dgn(S,m)}y^{\cd(S)}.
\end{align}

\begin{proposition}\label{prop:positive_schroder}
	For $m,n,t>0$, we have
	\begin{align*}
		\PPol_{m,n,t}(x,y) = \NPol_{m,n,t}\left(x+1,\frac{xy}{x+1}\right).
	\end{align*}
\end{proposition}
\begin{proof}
	Let us first expand the right-hand side of the equation in the statement.  We get
	\begin{displaymath}
		\NPol_{m,n,t}\left(x+1,\frac{xy}{x+1}\right) = \sum_{P\in\Dyck_{m,n,t}}(x+1)^{\val(P)-\ret(P)}x^{\ret(P)}y^{\ret(P)}.
	\end{displaymath}
	On the other hand, let $P\in\Dyck_{m,n,t}$ and fix a subset $A_{1}\subseteq\valset(P)\setminus\retset(P)$.  That is, $A_{1}$ contains a selection of the non-return valleys of $P$.  If we apply the bijection from Lemma~\ref{lem:asso_schroder_bijection} to the pair $(P,A_{1}\uplus\retset(P))$, then we obtain a Schr{\"o}der path $S\in\Schroder_{m,n,t}$ without returns, which satisfies
	\begin{align*}
		\dgn(S) & = \lvert A_{1}\rvert+\ret(P),\\
		\cd(S) & = \ret(P)
	\end{align*}
	Moreover, $S$ does not touch the line $x=my$, because $P$ touches this line only in its returns and when passing to $S$ all of these returns were converted to diagonal steps.  This states in fact that $S\in\Schroder_{m,n,t}^{+}$.  Conversely, if $S\in\Schroder_{m,n,t}^{+}\subseteq\Schroder_{m,n,t}$, then it is ensured by Lemma~\ref{lem:asso_schroder_bijection} that converting diagonal steps to valleys produces an element of $\mathfrak{F}_{m,n,t}$.

	By putting these observations together, we obtain
	\begin{align*}
		\sum_{P\in\Dyck_{m,n,t}}(x+1)^{\val(P)-\ret(P)} x^{\ret(P)}y^{\ret(P)} & = \sum_{\substack{P\in\Dyck_{m,n,t}\\A_{1}\subseteq\valset(P)\setminus\retset(P)}}x^{\lvert A_{1}\rvert+\ret(P)}y^{\ret(P)}\\
		& = \sum_{S\in\Schroder_{m,n,t}^{+}}x^{\dgn(S)}y^{\cd(S)}\\
		& = \PPol_{m,n,t}(x,y).\qedhere
	\end{align*}
\end{proof}

\begin{corollary}\label{cor:positive_schroder_form}
	For $m,n,t>0$, we have
	\begin{multline*}
		\PPol_{m,n,t}(x,y) = \sum_{a=0}^{n-t}\sum_{b=0}^{a}\Biggl(\frac{t+a}{n}-\frac{m(a-b)}{mn-b-1}\Biggr)\binom{mn-b-1}{a-b}\binom{mn+n-t-a-1}{n-t-a}x^{a}y^{b}.
	\end{multline*}
\end{corollary}
\begin{proof}
	This is a direct computation using Theorem~\ref{thm:mnt_narayana} and Proposition~\ref{prop:positive_schroder}.
\end{proof}

\begin{proposition}\label{prop:positive_schroder_mdiv}
	For $m,n,t>0$, we have
	\begin{align*}
		\PPol_{m,n,t}^{(m)}(x,y) = \APol_{m,n,t}\left(x+1,\frac{xy}{x+1}\right).
	\end{align*}
\end{proposition}
\begin{proof}
	This is essentially verbatim to the proof of Proposition~\ref{prop:positive_schroder} but restricting ourselves to $m$-valleys and $m$-diagonals.
\end{proof}

\begin{corollary}\label{cor:positive_schroder_mdiv_form}
	For $m,n,t>0$, we have
	\begin{displaymath}
		\PPol_{m,n,t}^{(m)}(x,y) = \sum_{a=0}^{n-t}\sum_{b=0}^{a}\Biggl(\frac{t+a}{n}-\frac{a-b}{n-b-1}\Biggr)\binom{n-b-1}{a-b}\binom{mn+n-t-a-1}{n-t-a}x^{a}y^{b}.
	\end{displaymath}
\end{corollary}
\begin{proof}
	This is a direct computation using Theorem~\ref{thm:mnt_htriangle} and Proposition~\ref{prop:positive_schroder_mdiv}.
\end{proof}

\begin{corollary}
	For $m,n,t>0$, we have
	\begin{align*}
		\PPol_{m,n,t}(x,y) & = \SPol_{m,n,t}\left(x,\frac{xy-1}{x}\right),\\
		\PPol_{m,n,t}^{(m)}(x,y) & = \SPol_{m,n,t}^{(m)}\left(x,\frac{xy-1}{x}\right).
	\end{align*}
\end{corollary}

\begin{proof}[Proof of Theorem~\ref{thm:positive_diagonal_enumeration}]
	This follows from Corollaries~\ref{cor:positive_schroder_form} and \ref{cor:positive_schroder_mdiv_form}.
\end{proof}

\begin{example}
	Let us consider the case $m=2$, $n=4$ and $t=2$.  We have already computed
	\begin{displaymath}
		\APol_{2,4,2}(x,y) = x^{2}y^{2} + x^{2}y + x^{2} + 5xy + 8x + 9
	\end{displaymath}
	in Example~\ref{ex:atriangle_242}.  Figure~\ref{fig:242_dyck_2valley} shows the sixteen $(2,2)$-Dyck paths of height $4$ with at least one $2$-valley together with their associated $2$-divisible small $(2,2)$-Schr{\"o}der paths.  Figure~\ref{fig:242_dyck_no_2valley} shows the nine $(2,2)$-Dyck paths of height $4$ without $2$-valleys.  These are themselves $2$-divisible small $(2,2)$-Schr{\"o}der paths without diagonal steps.  The reader is invited to check that there are $25$ paths without diagonal steps, eleven paths with one non-cornered diagonal step, eight paths with one cornered diagonal step, one path with two non-cornered diagonal steps, one path with two diagonal steps of which one is cornered and one path with two cornered diagonal steps.  This gives
	\begin{displaymath}
		\SPol_{2,4,2}^{(2)}(x,y) = x^{2}y^{2} + x^{2}y + x^{2} + 8xy + 11x + 25 = \APol_{2,4,2}\left(x+1,\frac{xy+1}{x+1}\right)
	\end{displaymath}
	in accordance with Proposition~\ref{prop:schroder_mdiags}.

	Moreover, in Figure~\ref{fig:242_dyck_2valley}, the non-positive $2$-divisible $(2,2)$-Schr{\"o}der paths are marked with a $*$.  Thus, the remaining paths and the paths in Figure~\ref{fig:242_dyck_no_2valley} constitute the 37-element set $\Schroder_{2,4,2}^{(2)+}$.  The reader may verify again that there are 18 such paths without diagonal steps, ten such paths with one non-cornered diagonal step, six such paths with one cornered diagonal step, one such path with two non-cornered diagonal steps, one such path with two diagonal steps of which one is cornered and one such path with two cornered diagonal steps.  This gives
	\begin{displaymath}
		\PPol_{2,4,2}^{(2)}(x,y) = x^{2}y^{2} + x^{2}y + x^{2} + 6xy + 10x + 18 = \APol_{2,4,2}\left(x+1,\frac{xy}{x+1}\right).
	\end{displaymath}
\end{example}

\begin{figure}
	\centering
	\includegraphics[page=16,width=\textwidth]{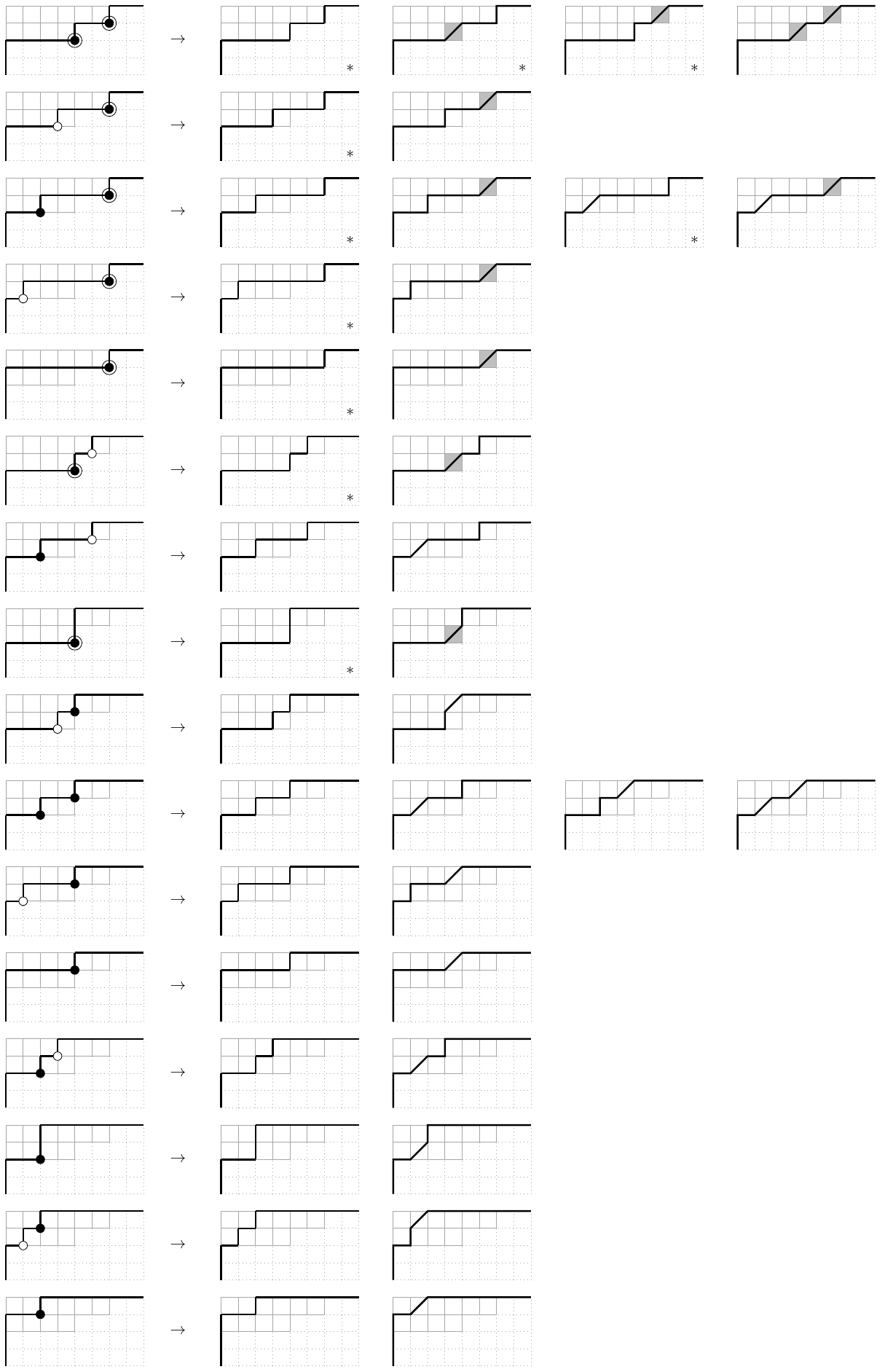}
	\caption{The 16 $(2,2)$-Dyck paths of height $4$ that have at least one $2$-valley together with the corresponding small $2$-divisible $(2,2)$-Schr{\"o}der paths. Cornered diagonals are highlighted.  The paths marked with a $*$ are non-positive.}
	\label{fig:242_dyck_2valley}
\end{figure}

\begin{figure}
	\centering
	\includegraphics[page=17,width=\textwidth]{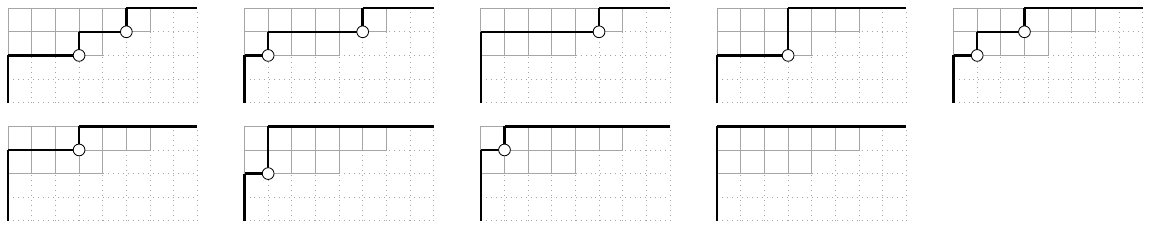}
	\caption{The nine $(2,2)$-Dyck paths of height $4$ that have no $2$-valley.}
	\label{fig:242_dyck_no_2valley}
\end{figure}

\section{Combinatorial reciprocity}\label{sec:poly_reciprocity}

\subsection{Basics}

Suppose that there is a counting function $f$, whose evaluation at a positive integer $n$
 yields the cardinality of the $n$-th member of some (combinatorial) collection.  In other words, there exists a collection of (interesting) sets $\{X_{n}\colon n\in\mathbb{N}\}$ such that $f(n)=\lvert X_{n}\rvert$.   In particular, the evaluation of $f$ at positive integers yields a non-negative integer.  This raises the question, if the evaluation of $f$ at negative integers produces---up to sign---an interesting number sequence as well.  If this is the case, \ie if there exists a collection of (interesting) sets $\{Y_{n}\colon n\in\mathbb{N}\}$ such that $f(-n)=(-1)^{n}\lvert Y_{n}\rvert$, then we say that $f$ exhibits a \defn{combinatorial reciprocity}.  See \cite{beck18combinatorial,stanley74combinatorial} for examples and further background.  

The prototypical example of a combinatorial reciprocity is the function $\mathcal{L}_{P}$ counting the lattice points contained in dilations of a lattice polytope $P$.  A celebrated result due to Ehrhart states that $(-1)^{t}\mathcal{L}_{P}(-t)$ counts the lattice points in the \emph{interior} of the $t$-fold dilation of $P$~\cite{ehrhart62sur}.  This kind of combinatorial reciprocity will be called \defn{Ehrhart reciprocity}.

\subsection{Combinatorial reciprocity between $(m,1)$-Dyck paths and $(m,1)$-Schr{\"o}der paths}

Computer experiments lead us to the discovery of the following combinatorial reciprocity involving $(m,1)$-Dyck paths and $(m,1)$-Schr{\"o}der paths.  We consider the following function
\begin{align*}
	d(m;n,t,a,b) & \defs \binom{n-b-2}{a-b}\binom{mn-t+1}{n-t-a} - m\binom{n-b-1}{a-b}\binom{mn-t}{n-t-a-1},
\end{align*}
which by Theorem~\ref{thm:mnt_htriangle} describes the coefficient of $x^{a}y^{b}$ in $\APol_{m,n,t}(x,y)$, respectively.  If we interpret this function as a polynomial in $m$, then we get a counting function as defined above.  In particular, $d(m,n,1,a,b)$ counts the number of $(m,1)$-Dyck paths of height $n$ with exactly $a$-many $m$-valleys and $b$-many returns.  This function exhibits a remarkable combinatorial reciprocity involving positive $(m,1)$-Schr{\"o}der paths.  Let us write
\begin{align*}
	p(m;n,t,a,b) & \defs \Biggl(\frac{t+a}{n}-\frac{a-b}{n-b-1}\Biggr)\binom{n-b-1}{a-b}\binom{mn+n-t-a-1}{n-t-a},
\end{align*}
for the coefficient of $x^{a}y^{b}$ of $\PPol_{m,n,t}^{(m)}(x,y)$.  

\begin{theorem}\label{thm:dyck_schroder_reciprocity}
	For $m,n>0$ and $0\leq a\leq n-1$, $0\leq b\leq a$, we have
	\begin{displaymath}
		d(-m,n,1,a,b) = (-1)^{n-1-a}p(m,n,1,a,b).
	\end{displaymath}
\end{theorem}
\begin{proof}
	This is a direct computation.
	
	We get
	\begin{align*}
		(-1)^{n-1-a} & d(-m;n,1,a,b)\\
			& = (-1)^{n-1-a}\Biggl(\binom{n-b-2}{a-b}\binom{-mn}{n-1-a}\\
				& \kern2cm + m\binom{n-b-1}{a-b}\binom{-mn-1}{n-a-2}\Biggr)\\
			& = (-1)^{n-1-a}\binom{n-b-2}{a-b}\binom{-mn}{n-1-a}\\
				& \kern2cm - (-1)^{n-a-2}m\binom{n-b-1}{a-b}\binom{-mn-1}{n-a-2}\\
			& = \binom{n-b-2}{a-b}\binom{mn+n-a-2}{n-1-a} - m\binom{n-b-1}{a-b}\binom{mn+n-a-2}{n-a-2}\\
			& = \Biggl(\frac{n-a-1}{n-b-1} - \frac{n-a-1}{n}\Biggr)\binom{n-b-1}{a-b}\binom{mn+n-a-2}{n-a-1}\\
			& = \Biggl(\frac{n-a-1}{n-b-1} - 1 + \frac{a+1}{n}\Biggr)\binom{n-b-1}{a-b}\binom{mn+n-a-2}{n-a-1}\\
			& = \Biggl(\frac{b-a}{n-b-1} + \frac{a+1}{n}\Biggr)\binom{n-b-1}{a-b}\binom{mn+n-a-2}{n-a-1}\\
			& = p(m;n,1,a,b).\qedhere
	\end{align*}
\end{proof}

In terms of the generating functions $\APol_{m,n,t}(x,y)$ and $\PPol_{m,n,t}^{(m)}(x,y)$, this reads as follows.

\begin{corollary}
	For $m,n>0$, we have
	\begin{displaymath}
		\PPol_{m,n,1}^{(m)}(x,y) = (-1)^{n-1}\APol_{-m,n,1}(-x,y).
	\end{displaymath}
\end{corollary}

The interpretation of the evaluation of $d(m;n,1,a,b)$ at negative values of $m$ in terms of generating function motivated us to plug in negative values for $m$ into all of our counting polynomials.  Since we know explicit formulas for all of them, we can directly compute the resulting polynomials.  We can then relate them to the original polynomials by appropriate sign-changes and variable substitutions.  It is an intriguing question to understand what these equations mean combinatorially.

\begin{proposition}\label{prop:all_reciprocities}
	For $m,n>0$, we have
	\begin{align*}
		\APol_{-m,n,1}(x,y) & = (-1)^{n-1}\APol_{m,n,1}\left(1-x,\frac{-xy}{1-x}\right),\\
		\BPol_{-m,n,1}(x,y) & = (-1)^{n-1}\BPol_{m,n,1}(1-x,1-y),\\
		\NPol_{-m,n,1}(x,y) & = (1-x)^{n-1}\NPol_{m,n,1}\left(\frac{x}{x-1},1-y\right),\\
		\FPol_{-m,n,1}(x,y) & = (-x)^{n-1}\FPol_{m,n,1}\left(\frac{1}{x},\frac{-y}{x}\right),\\
		\FPol_{-m,n,1}^{(m)}(x,y) & = (x+1)^{n-1}\FPol_{m,n,1}^{(m)}\left(\frac{-x}{x+1},\frac{y-x}{x+1}\right),\\
		\SPol_{-m,n,1}(x,y) & = (-x)^{n-1}\SPol_{m,n,1}\left(\frac{1}{x},-xy\right),\\
		\SPol_{-m,n,1}^{(m)}(x,y) & = (-1)^{n-1}\SPol_{m,n,1}^{(m)}\left(-x-1,\frac{xy+2}{x+1}\right),\\
		\PPol_{-m,n,1}(x,y) & = (-x)^{n-1}\PPol_{m,n,1}\left(\frac{1}{x},x(1-y)+1\right),\\
		\PPol_{-m,n,1}^{(m)}(x,y) & = (-1)^{n-1}\PPol_{m,n,1}^{(m)}\left(-1-x,\frac{xy}{1+x}\right).
	\end{align*}
\end{proposition}
\begin{proof}
	We can either compute these relations directly, using Theorems~\ref{thm:mnt_narayana},~\ref{thm:mnt_htriangle}, and \ref{thm:mnt_htriangle_var}, as well as Corollaries~\ref{cor:dyck_ftriangle_form},~\ref{cor:m_dyck_ftriangle_form},~\ref{cor:schroder_diags_form},~\ref{cor:schroder_mdiags_form},~\ref{cor:positive_schroder_form} and \ref{cor:positive_schroder_mdiv_form}.

	Alternatively, we may verify these relations using the transformations from Lemmas~\ref{lem:nb_relation},~\ref{lem:na_relation}, and \ref{lem:ab_relation}, as well as Propositions~\ref{prop:m_dyck_ftriangle_all},~\ref{prop:m_dyck_ftriangle},~\ref{prop:schroder_diags},~\ref{prop:schroder_mdiags},~\ref{prop:positive_schroder} and \ref{prop:positive_schroder_mdiv}.
\end{proof}

\begin{remark}
	It follows from the origin of $\APol_{m,n,t}(x,y)$ mentioned in Remark~\ref{rem:atriangle_origin} that when $t=1$, this is precisely the type-$A$ \emph{$H$-triangle} defined by Chapoton and Armstrong; see \cite[Definition~5.3.1]{armstrong09generalized} and \cite[Section~6]{chapoton06sur}.  From this perspective it is not surprising that 
	the relation 
	\begin{displaymath}
		\APol_{-m,n,1}(x,y) = (-1)^{n-1}\APol_{m,n,1}\left(1-x,\frac{-xy}{1-x}\right)
	\end{displaymath}
	from Proposition~\ref{prop:all_reciprocities}
	is the same as the type-$A$ case of \cite[Corollary~3]{thiel14htriangle}.  Interestingly, we discovered the reciprocities in Proposition~\ref{prop:all_reciprocities} before realizing that they are equivalent to that in \cite[Corollary~3]{thiel14htriangle}.
\end{remark}

\subsection{Ehrhart reciprocity}

In Remark~\ref{rem:positive_narayana}, we have considered the (positive) Fu{\ss}--Catalan numbers defined by
\begin{align*}
	\Cat(m,n) = \frac{1}{mn+1}\binom{(m+1)n}{n}\quad\text{and}\quad\Cat^{+}(m,n) = \frac{1}{mn-1}\binom{(m+1)n-2}{n}.
\end{align*}
As was pointed out in \cite[Section 7]{atz-pc-06}, these numbers are related by an instance of \emph{Ehrhart reciprocity}, which manifests in the equality:
\begin{equation}\label{eq:catalan_reciprocity}
	\Cat(-m,n) = (-1)^{n-1}\Cat^{+}(m-1,n).
\end{equation}
The key fact behind this reciprocity is 
that the dominant regions of the Catalan arrangement, which are enumerated by $\Cat(m,n)$,
 are in bijection with the integer points of a certain simplex (see \cite[Section 4]{Ath_RGCN05}).
 An analogous bijection holds for the bounded dominant regions, which are enumerated by $\Cat^+(m,n)$ \cite{atz-pc-06}.
By design, we have
\begin{align*}
	\Cat(m,n) & = \NPol_{m,n,1}(1,1) = \APol_{m,n,1}(1,1) = \BPol_{m,n,1}(1,1),\\
	\Cat^{+}(m,n) & = \NPol_{m,n,1}(1,0) = \APol_{m,n,1}(1,0) = \BPol_{m,n,1}(1,0).
\end{align*}
Therefore, \eqref{eq:catalan_reciprocity} can be rewritten as
\begin{align*}
	\NPol_{-m,n,1}(1,1) & = (-1)^{n-1}\NPol_{m-1,n,1}(1,0),\\
	\APol_{-m,n,1}(1,1) & = (-1)^{n-1}\APol_{m-1,n,1}(1,0),\\
	\BPol_{-m,n,1}(1,1) & = (-1)^{n-1}\BPol_{m-1,n,1}(1,0).
\end{align*}

Motivated by the above reciprocities, we show that specializing the relation 
\begin{equation}\label{eq:n_reciprocity}
	\NPol_{m,n,1}(x+1,y+1) = (-x)^{n-1}\NPol_{-m,n,1}\left(\frac{x+1}{x},-y\right),
\end{equation}
which follows from Proposition~\ref{prop:all_reciprocities}, to $y=-1$ is also an instance of Ehrhart reciprocity. Before proceeding,  we need to  briefly revise some facts on Ehrhart reciprocity. If $P\subseteq \mathbb R^n$ is a polytope with rational vertices, for any $t\in \mathbb N$ 	we denote by $tP$ the $t$-fold dilation of $P$.
The function $\mathcal L_{P}(t) \defs |\mathbb Z^n \cap t P|$ is a quasipolynomial in $t$, called the \defn{Ehrhart (quasi)polynomial} of $P$.
Similarly,  for the interior  $P^{\circ}$ of $P$ we define 
$\mathcal L_{P^{\circ}}(t) \defs |\mathbb Z^n \cap t P^{\circ}|$, which is again a quasipolynomial in $t$.  The Ehrhart reciprocity states that
\begin{equation}
	\mathcal L_{P} (-t) = (-1)^{\dim(P)} \mathcal L_{P^{\circ}} (t).
	\label{ER}
\end{equation}

We now apply this to the $(n-1)$-dimensional simplex 
\begin{displaymath}
	\Ao \defs \{{\bf x}\in \mathbb R^n\colon x_1\geq x_2 \geq \cdots \geq x_n \geq x_1-1\}\cap \{{\bf x}\in \mathbb R^n\colon x_1+\cdots+x_n=0\}.
\end{displaymath}
In view of \cite{Ath_RGCN05,atz-pc-06} we have 	that
\begin{align}
	\Nar(m,n,k) = \frac{1}{mn+1} \binom{n-1}{n-k-1}\binom{mn+1}{n-k}
	\label{Narnum}
\end{align}
is the number of integer points in $\mathbb Z^n$ lying in exactly $k$ walls of 
$(mn+1)\Ao$ for $k=0,\ldots,n-1$. Likewise, 
\begin{align}
 \Nar^+(m,n,k) = \frac{1}{mn-1} \binom{n-1}{n-k-1}\binom{mn-1}{n-k}
	\label{pNarnum}
\end{align}
is the number of integer  points lying in exactly $k$ walls of $(mn-1)\Ao$ for $k=0,\ldots,n-1$.
In terms of our generating functions, we have 
\begin{align*}
	\NPol_{m,n,1}(x,1) = \sum_{i=0}^{n-1} \Nar(m,n,i) x^{n-1-i}, \\
	\NPol_{m,n,1}(x,0) =  \sum_{i=0}^{n-1} \Nar^+(m,n,i) x^{n-1-i},
\end{align*}
\hm{see also Remarks~\ref{rem:fuss_narayana} and \ref{rem:positive_narayana}.}

\begin{theorem}\label{nar_ehrhart_reciprocity}
The recursive relation 
\begin{equation}
	\Nar^+(-m,n,n-1-k) = (-1)^{k} \sum_{i=0}^{k} \Nar(m,n,n-1-i) \binom{n-1-i}{n-1-k}
	\label{nar_rec}
\end{equation}
is an instance of Ehrhart reciprocity. 
\end{theorem}
\begin{proof}
As we mentioned earlier,
$	\Nar^+(m,n,n-1-k) $ is equal to the number of integer points 
lying in {\em exactly} $n-1-k$ facets of  $(mn-1)\Ao$. 
Since $(mn-1)\Ao$ is an $(n-1)$-dimensional simplex,
the intersection of $n-1-k$ facets uniquely defines a $k$-face. 
Therefore, the points lying in exactly $n-1-k$ facets of the simplex  $(mn-1)\Ao$
are those lying in the {\em interior}  $F^\circ$  of the $k$-dimensional face $F$ of $(mn-1)\Ao$
defined by the intersection of these facets. 
We therefore have
\begin{align*}
	\Nar^+(m,n,n-1-k)  =\sum_{
		\substack{	F\text{ face of } \Ao \\ \dim(F)=k }} \mathcal L_{F^{\circ}}(mn-1).
\end{align*}
Replacing $m$ by $-m$ and applying Ehrhart reciprocity, we get
\begin{align*}
	\Nar^+(m,n,n-1-k) 
	& =\sum_{
		\substack{	F\text{ face of } \Ao \\ \dim(F)=k }} \mathcal L_{F^{\circ}}(-mn-1)\\
	& = \sum_{
		\substack{	F\text{ face of } \Ao \\ \dim(F)=k }} (-1)^k \mathcal L_{F}(mn+1)\\
	& = (-1)^k \sum_{
		\substack{	F\text{ face of } \Ao \\ \dim(F)=k }}  \mathcal L_{F}(mn+1).
\end{align*} 
We next claim that
\begin{align}
	\sum_{i=0}^{k} \Nar(m,n,n-1-i) \binom{n-1-i}{n-1-k} = 
	\sum_{
		\substack{	F\text{ face of } \Ao \\ \dim(F)=k }}  \mathcal L_{F}(mn+1). \label{count_points}
\end{align}
The right-hand side of \eqref{count_points} counts the integer points of $(mn+1)\Ao$
over all $k$-dimensional (closed) faces of $(mn+1)\Ao$. 
For $i=0,\ldots,k$, the left-hand side of \eqref{count_points} counts 
pairs $(p,F)$ where $p$ in an  integer point of $(mn+1)\Ao$ 	lying in exactly $n-1-i$ facets of $(mn+1)\Ao$ and $F$ is a subset of $n-1-k$ of these facets.
Now, if a point lies in exactly $n-1-i$ facets of the simplex  $(mn+1)\Ao$, 
then it  lies in exactly $\binom{n-1-i}{n-1-k}$ $k$-dimensional faces of $(mn+1)\Ao$.
Therefore, each integer point on the left-hand side of \eqref{count_points} is counted as many times as
the number of  $k$-dimensional faces it belongs to. This is precisely the counting on the right hand side of  \eqref{count_points}, and this completes our proof. 
\end{proof}

The following corollary is essentially the content of Theorem \ref{nar_ehrhart_reciprocity}
in terms of generating functions. 
\begin{corollary}\label{nar_gf_ehrhart_reciprocity}
	The relation 
	\begin{equation*}
		\NPol_{m,n,1}(x+1,0) = (-x)^{n-1}\NPol_{-m,n,1}\left(\frac{x+1}{x},1\right).
		\label{rel:er}
	\end{equation*}
	is an instance of Ehrhart reciprocity. 
\end{corollary}
\begin{proof}

	The relation in the statement of the Lemma  is equivalent to 
	\begin{align}
		\NPol_{m,n,1}(x,0) = (1-x)^{n-1} \NPol_{-m,n,1}\left(\frac{x}{x-1},1\right).
		\label{Nrec}
	\end{align}
	If we expand  right-hand side of the  above, we have
	\begin{align}
		(1-x)^{n-1} \NPol_{-m,n,1}\left(\frac{x}{x-1},1\right)  & =
		(1-x)^{n-1} \sum_{i=0}^{n-1}	\Nar(-m,n,n,i) \Bigl(\frac{x}{x-1}\Bigr)^{n-1-i} \notag \\
		& = (1-x)^{n-1} \sum_{i=0}^{n-1}	\Nar(-m,n,n-1-i) \Bigl(\frac{x}{x-1}\Bigr)^{i} \notag \\
		& =   (-1)^{n-1} \sum_{i=0}^{n-1}	\Nar(-m,n,n-1-i) x^{i} (x-1)^{n-1-i} \notag \\
			& =   (-1)^{n-1} \sum_{i=0}^{n-1}	\Nar(-m,n,n-1-i) x^{i}\label{2}\\
				& \kern2cm \times \sum_{j=0}^{n-1-i}
		\binom{n-1-i}{j}
		x^j (-1)^{n-1-i-j}. \notag
	\end{align}
	Thus, equating the coefficients of $x^{k}$ in \eqref{Nrec}  and \eqref{2},
 it remains  to prove that 
	\begin{align*}
		\Nar^+(-m,n,n-1-k) = (-1)^{k} \sum_{i=0}^{k} \Nar(m,n,n-1-i) \binom{n-1-i}{n-1-k},
	\end{align*}
which is precisely the relation in Theorem \ref{nar_ehrhart_reciprocity}.
	\end{proof}

\begin{remark}\label{remark:ehr_general}
	It is not clear to us how to extend our result to the general case, i.e., explain  relation  \eqref{eq:n_reciprocity}  as an instance of Ehrhart reciprocity. The main reason it that the bijection between  dominant regions and integer points of  dilations of $\Ao$ does not distinguish between separating walls and simple separating walls (see also Proposition \ref{cor:Apol_regions}). 	Therefore, it is not clear how our bivariate distribution is carried over to integer points of  the simplex $\Ao$. 
\end{remark}

\section{Refined counting in hyperplane arrangements}
\label{sec:arrangements}

The main result of this section is a  combinatorial interpretation of the 
bivariate generating function $x^{n-1}\APol_{m,n,1}\left(\frac{x+1}{x},\frac{y+1}{x+1}\right)$
in the setting of the Catalan arrangement. 
The $m$-Catalan arrangement corresponding to the Coxeter group  of type $A_{n-1}$ (or equivalently the symmetric group $S_n$), here denoted by $\mathcal Cat^m(n)$, is defined by the hyperplanes 
$x_i-x_j = k$ with  $-m \leq k \leq m$. 
The set of  \defn{dominant regions} ${\sf Dom}(m,n)$  consists of the regions of $\mathcal Cat^m(n)$  lying in the cone $x_1>x_2>\cdots >x_n$, which is known as the dominant chamber. 
A hyperplane of $\mathcal Cat^m(n)$ is  a \defn{wall}  of a region $R$ if it is a supporting hyperplane of $R$. 
A wall is a \defn{separating wall}  of $R$ if it separates $R$ from the origin. 

As we have already mentioned in Section \ref{sec:valley_mnt}, if we count the number of dominant regions according to the number of  their separating walls of type $x_i-x_j=m$, we obtain the Narayana numbers. 
The $H$-triangle of the Catalan arrangement refines the above counting
 of the dominant regions by   distinguishing separating walls 
 and  \defn{simple separating walls}, i.e.,  separating walls of type  $x_i-x_{i+1}=m$. 
 Such a distinction makes sense, since the regions with no simple separating walls are those which are bounded, whereas 
the bigger the number of simple separating walls of a region is, the bigger is its number of unbounded facets.

	Combining the transformation of the $H$-triangle to the $M$-triangle obtained by combinining \eqref{eq:f_to_h} and \eqref{eq:f_to_m}, see
	\cite[Corollary 2]{thiel14htriangle}, and the explicit computations in \cite[Section 4]{krattenthaler22rank} we conclude that the $H$-triangle in \cite{thiel14htriangle}
	is exactly $\APol_{m,n,1}(x,y)$. More precisely, we have the following. 
	\begin{proposition}[\cite{krattenthaler22rank,thiel14htriangle}]
		\label{cor:Apol_regions}
		The following are equal: 
		\begin{enumerate}[\rm(i)]
			\item 	the  number of dominant regions 
			in the $m$-Catalan arrangement having 
			$a$-many separating walls of type $x_i-x_j=m$ and $b$-many simple separating walls of type $x_{i}-x_{i+1}=m$,
			\item the number of 
      Dyck paths $P\in \Dyck_{m,n,1}$ with  $a$-many $m$-valleys and $b$-many returns,
      \item the coefficient of $x^ay^b$ of $\APol_{m,n,1}(x,y)$.
		\end{enumerate}
		\end{proposition}
	For each dominant  region $R$ of $\mathcal Cat^m(n)$, let us  denote by $\swset(R)$ the set of separating walls of type $x_i-x_j=m$ 	of $R$ and by  $\sswset(R)$ its subset of simple separating walls. Let us write  
	$\sw(R)=|\swset(R)|$ and 	$\ssw(R)=|\sswset(R)|$. 
	Then, Proposition \ref{cor:Apol_regions} states that 
		\begin{align}
			\label{AisH}
		 \APol_{m,n,1}(x,y) = \sum_{R\in {\sf Dom}(m,n)} x^{\sw(R)} y ^{\ssw(R)}.
	\end{align}

	\begin{remark}
		\label{pak-stanley}
Surprisingly, the  well-known bijection of Pak and Stanley
  between regions of the $m$-Shi arrangement and 
	$m$-parking functions (see \cite{Stan_HAPFTI98}), when restricted to dominant
	regions does not preserve the above refinement. In other words, 
	although return points correspond to simple separating walls, 	
	 $m$-valleys do not always correspond to separating walls of type $x_i-x_j=m$. 
	It would be interesting to find a bijection that agrees with Proposition \ref{cor:Apol_regions}.
		\end{remark}

\begin{figure}[h]
	\label{fig:flats1}
	\centering
	\includegraphics[page=13,width=0.5\textwidth]{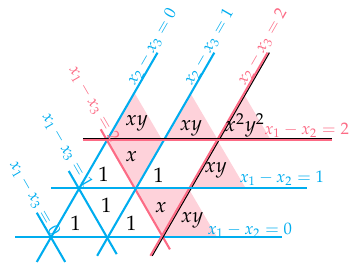}
	\caption{$\APol_{2,3,1}(x,y)=x^2y^2+4xy+2x+5$.}
\end{figure}
\begin{figure}[h]
	\includegraphics[page=14,width=0.5\textwidth]{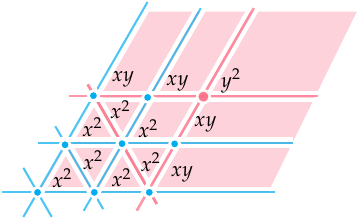}
		\hspace{-0.5cm}
	\includegraphics[page=15,width=0.5\textwidth]{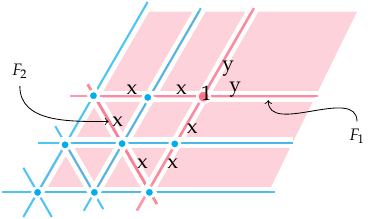}
	\caption{  $x^{2}\APol_{2,3,1}(\frac{x+1}{x},\frac{y+1}{x+1})=7 x^{2} + 4xy + y^{2} + 6x + 2y + 1$.}
	\label{fig:flats2}
\end{figure}

The regions in a hyperplane arrangement are polytopes or unbounded polyhedra, whose faces
 are called  \defn{flats} of the arrangement. 
Alternatively, if  $\mathcal H$ is a hyperplane arrangement,  a flat in $\mathcal H$ is any non-empty intersection of type
 \begin{align*}
 	\label{flats}
 \bigcap_{
 	\substack{	H \in \mathcal H \\ \epsilon \in \{-1,0,1\}} } H^{\epsilon}
  \end{align*}
 where  $H^{\pm}$ is the positive (resp. negative) open half-space  corresponding to $H\in \mathcal H$ and $H^{0}=H\in \mathcal H$. 

In our setup, the sets $H^{\epsilon}$ fall into one of the following cases:
\begin{displaymath}
\begin{cases}  x_i-x_j>0  \\
      x_i-x_j >k \text{ or } x_i-x_j < k  \text{ for } k=1,\ldots,m-1\\
      x_i-x_j >m \text{ or } x_i-x_j < m \text{ or } x_i-x_j =m.   
     \end{cases}
\end{displaymath}
We denote by $\flats$ the non-empty intersections of such halfspaces/hyperplanes.
In other words, $H^{\epsilon}$  can be any open halfspace  in the dominant chamber or any   hyperplane  of type $x_i-x_j=m$. 
Flats involving intersections with  hyperplanes of type  $x_i-x_j=k$ with $k<m$  do not belong to the set $\flats$ (see Figure \ref{fig:flats2}).  

\medskip

The key property of the set  $\flats$ is that each   $F\in\flats$
can be uniquely assigned to a dominant region:  the one in whose separating walls it is contained.  
In other words, although a flat  $F \in \flats$ may be a face of more than one regions, 
there is a unique region which contains $F$ in its set of separating walls. 
This is formally stated in the next lemma. 
\begin{theorem}\cite[Theorem 5.1]{atz-pc-06}
	\label{lem:flats}
	There is a bijection between  $\flats$ and 
the set 
\begin{align}
 \mathfrak R_{m,n} = \bigl\{(R,S)\colon R \in  {\rm dom}(m,n) \text{ and } S\subseteq \swset(R)\bigr\}
\end{align}	
where $\swset(R)$ is the set of separating walls of $R$ of type $x_i-x_j=m$. 
\end{theorem}

The bijection $\tau\colon \mathfrak R_{m,n}\to\flats$ 
of the above theorem is the following.  
If   $(R,S)\in \mathfrak R_{m,n}$ then  the flat $F=\tau(R,S)$
is defined as the  intersection of all halfspaces  $H^{\epsilon}$  with $R\subseteq H^{\epsilon}$
except those $H^{\epsilon}$ for which $H\in S$, in which case they become $H^0$.

The set of (simple) separating walls of a flat can be defined  in a way entirely  analogous to that of  a region.  In view of the bijection above,   
if $(R,S) \in \mathfrak R_{m,n}$ is the pair corresponding to the flat $F\in \flats$ then the set of 
\defn{(simple) separating walls} of $F$ is the set  of (simple) separating walls of $R$
minus  those for which $F\subseteq H^0$.
More precisely,  if $S_2=S\setminus \sswset(R)$ then 
$\swset(F)=\swset(R)\setminus S$ and 
$\sswset(F)=\sswset(R)\setminus S_2$. 

\begin{example}
For instance, in Figure \ref{fig:flats2},
the hyperplane $x_2-x_3=2$ is a simple separating wall of the  flat $F_1$, whereas the hyperplane 
$x_1-x_2=2$ is not a separating wall of $F_1$ because $F_1$ is contained in it. 
The flat $F_2$ has no  separating walls of type $x_i-x_j=2$.
\end{example}

\begin{proposition}
	\label{prop:FtriangleArrang}
	We have that 
	$$x^{n-1} \APol_{m,n,1}\left(\frac{x+1}{x},\frac{y+1}{x+1}\right)
	= \sum_{F\in \flats} x^{\dim(F)-\ssw(F)} y^{\ssw(F)},
	$$ 
	where $\ssw(F)=|\sswset(F)|$ is the number of simple separating walls of type $x_i-x_{i+1}=m$ of the flat $F\in \mathcal F_{m,n}$.
\end{proposition}
\begin{proof}

		For each dominant region $ R \in  {\sf Dom(m,n)}$ the set $\swset(R)$ can be partitioned as $S_1 \uplus S_2$ where $S_2 =\sswset(R)$ is the set of simple separating walls of $R$ of type $x_i-x_{i+1}=m$ and $S_1=\swset(R)\setminus S_2$. 
	In view of Corollary \ref{cor:Apol_regions} we have
	\begin{align}
		x^{n-1}   \APol_{m,n,1}\left(\frac{x+1}{x},\frac{y+1}{x+1}\right) &  = 
		x^{n-1}  \sum_{R\in {\sf Dom}(m,n)}  \Bigl(\frac{x+1}{x}\Bigr)^{\sw(R)}\Bigl(\frac{y+1}{x+1}\Bigr)^{\ssw(R)} \notag\\
		& =  \sum_{R\in {\sf Dom}(m,n} (x+1)^{\sw(R)-\ssw(R)}x^{n-1-\sw(R)} (y+1)^{\ssw(R)} \notag\\
		& =  \sum_{R\in {\sf Dom}(\mathcal Cat^m(n))} (x+1)^{|S_1|}x^{n-1-|S_1|-|S_2|} (y+1)^{|S_2|} \notag\\
		& =  \sum_{R\in {\sf Dom}(m,n)} \sum_{i,j} \binom{|S_1|}{i}
		x^{|S_1|-i}x^{n-1-|S_1|-|S_2|} \binom{|S_2|}{j} y^{|S_2|-j} \notag \\
	& =  \sum_{R\in {\sf Dom}(m,n)} \sum_{i,j} \binom{|S_1|}{i}
x^{n-1-|S_2|-i} \binom{|S_2|}{j} y^{|S_2|-j} \notag \\
		& = 
		\sum_{R\in {\sf Dom}(m,n)}
		\sum_{A_1\subseteq S_1} \sum_{A_2\subseteq S_2} x^{n-1-|S_2|-|A_1|} y^{|S_2|-|A_2|} \label{to_flats1} \\
		 & = \sum_{F \in\flats} x^{\dim(F)-\ssw(F)} y^{\ssw(F)} \label{to_flats2},
	\end{align}
where to go from \eqref{to_flats1} to \eqref{to_flats2} we used the bijection  $ \tau^{-1}(F)=(R,A_1\uplus S_2)$
as well as  the fact that $\dim(F) = n-1-|A_1|-|A_2|$
and $\ssw(F) = \ssw(R)-|A_2| = |S_2|-|A_2|$.
\end{proof}

\begin{remark}
	Propositions \ref{prop:m_dyck_ftriangle} and \ref{prop:FtriangleArrang} suggest that there should exist a bijection between $m$-divisible $(m,1)$-Schr\"oder paths  counted according to dimension and returns and flats in $\flats$ counted according to dimension and simple separating walls.  A bijective explanation of Proposition~\ref{cor:Apol_regions} would immediately imply the existence of such a  map.
\end{remark}


\clearpage

\appendix

\section{Details of the generating function proofs}

\subsection{Proof of Proposition~\ref{prop:m_dyck_valley_noret}}
\label{app:prop_m_dyck_valley_noret}

\nposform*
\begin{proof}[Detailed proof]
	We apply Lemma~\ref{lem:lagrange_burmann} to the functions $f(z)=\frac{z}{(z+1)(xz+1)^{m}}$ and $g_{+}(z)=(z+1)(xz+1)^{m-1}$.  This yields
	\begin{align*}
		\NPol_{m,n,1}(x,0)
		& = \frac{1}{n-1}\langle z^{-1}\rangle\Biggl( (xz+1)^{m-2}\Bigl(xz+1 + (m-1)x(z+1)\Bigr)\\
			& \kern5cm \times\left(\frac{z}{(z+1)(xz+1)^{m}}\right)^{-(n-1)}\Biggr)\\
		& = \frac{1}{n-1}\langle z^{n-2}\rangle\Biggl( (z+1)^{n-1}(xz+1)^{mn-1} + (m-1)x(z+1)^{n}(xz+1)^{mn-2}\Biggr)\\
		& = \frac{1}{n-1}\langle z^{n-2}\rangle\Biggl( \sum_{k=0}^{n-1}\sum_{\ell=0}^{mn-1}\binom{n-1}{k}\binom{mn-1}{\ell}x^{\ell}z^{k+\ell}\\
			& \kern3cm + (m-1)\sum_{k=0}^{n}\sum_{\ell=0}^{mn-2}\binom{n}{k}\binom{mn-2}{\ell}x^{\ell+1}z^{k+\ell}\Biggr)\\
		& = \frac{1}{n-1} \sum_{r=0}^{n-1}\Biggl(\binom{n-1}{r+1}\binom{mn-1}{r} + (m-1)\binom{n}{r+1}\binom{mn-2}{r-1}\Biggr) x^{r}\\
		& = \sum_{r=0}^{n-1}\frac{mn-r-1}{n(mn-1)} \binom{n}{r+1}\binom{mn-1}{r}x^{r}\\
		& = \sum_{r=0}^{n-1}\frac{1}{n}\binom{n}{r+1}\binom{mn-2}{r}x^{r}.\qedhere
	\end{align*}
\end{proof}

\subsection{Proof of Theorem~\ref{thm:mnt_narayana}}
\label{app:mnt_narayana}

\nform*
\begin{proof}[Detailed proof]
	We apply Lemma~\ref{lem:lagrange_burmann} to the functions $f(z)=\frac{z}{(z+1)(xz+1)^{m}}$ and $g(z)=(z+1)^{s+1}(xz+1)^{m(s+t)-(s+1)}$ to obtain the claimed formula.
	
	We explicitly record the derivative of $g(z)$:
	\begin{multline}\label{eq:ghat_deriv}
		g(z)' = (z+1)^{s}(xz+1)^{m(s+t)-s-2}\\
		\times\Bigl((s+1)(xz+1) + \bigl(m(s+t)-s-1\bigr)x(z+1)\Bigr).
	\end{multline}
	We get
	\begin{align*}
		\NPol_{m,n,t} & (x,y) \overset{\parbox{1cm}{\centering\tiny\eqref{eq:mnt_narayana_gf},\eqref{eq:dyck_val_ret_gf_B}}}{=} \sum_{s=0}^{\infty}(xy)^{s} \langle z^{n-t-s}\rangle g\bigl(F^{(1)}(x,1;z)\bigr)\\
		& \overset{\parbox{1cm}{\centering\tiny\text{Lem.~}\ref{lem:lagrange_burmann}}}{=} (xy)^{n-t} + \sum_{s=0}^{n-t-1}\frac{(xy)^{s}}{n-t-s}\langle z^{-1}\rangle g'(z)\left(\frac{z}{(z+1)(xz+1)^{m}}\right)^{-(n-t-s)}\\
		& \overset{\parbox{1cm}{\centering\tiny\eqref{eq:ghat_deriv}}}{=} (xy)^{n-t} + \sum_{s=0}^{n-t-1}\frac{(xy)^{s}}{n-t-s}\langle z^{n-t-s-1}\rangle\\
			& \kern3cm \times\Biggl((s+1)(z+1)^{n-t}(xz+1)^{mn-s-1}\\
			& \kern3.5cm  + \Bigl(m(s+t)-s-1\Bigr)x(z+1)^{n-t+1}(xz+1)^{mn-s-2}\Biggr)\\
		& \overset{\parbox{1cm}{\centering\tiny~}}{=} (xy)^{n-t} + \sum_{s=0}^{n-t-1}\frac{(xy)^{s}}{n-t-s}\langle z^{n-t-s-1}\rangle\\
			& \kern3cm \times \Biggl((s+1)\sum_{k=0}^{n-t}\sum_{\ell=0}^{mn-s-1}\binom{n-t}{k}\binom{mn-s-1}{\ell}x^{\ell}z^{k+\ell}\\
			& \kern3.5cm  + \Bigl(m(s+t)-s-1\Bigr)\\
			& \kern4cm \times\sum_{k=0}^{n-t+1}\sum_{\ell=0}^{mn-s-2}\binom{n-t+1}{k}\binom{mn-s-2}{\ell}x^{\ell+1}z^{k+\ell}\Biggr)\\
		& \overset{\parbox{1cm}{\centering\tiny~}}{=} (xy)^{n-t} + \sum_{s=0}^{n-t-1}\frac{(xy)^{s}}{n-t-s}\\
			& \kern3cm \times\Biggl((s+1)\sum_{r=0}^{n-t-s-1}\binom{n-t}{n-t-s-r-1}\binom{mn-s-1}{r}x^{r}\\
			& \kern3.5cm  + \Bigl(m(s+t)-s-1\Bigr)\\
			& \kern4cm \times\sum_{r=0}^{n-t-s-1}\binom{n-t+1}{n-t-s-r-1}\binom{mn-s-2}{r}x^{r+1}\Biggr)\\
		& \overset{\parbox{1cm}{\centering\tiny~}}{=} (xy)^{n-t} + \sum_{s=0}^{n-t-1}\frac{(xy)^{s}}{n-t-s}\sum_{r=0}^{n-t-s}x^{r}\Biggl((s+1)\binom{n-t}{r+s+1}\binom{mn-s-1}{r}\\
			& \kern3.5cm  + \Bigl(m(s+t)-s-1\Bigr)\binom{n-t+1}{r+s+1}\binom{mn-s-2}{r-1}\Biggr).
	\end{align*}
	
	We will now simplify the inner summand:
	\begin{align*}
		S(m,n,t,r,s) & \defs (s+1)\binom{n-t}{r+s+1}\binom{mn-s-1}{r}\\
		& \kern1cm + \Bigl(m(s+t)-s-1\Bigr)\binom{n-t+1}{r+s+1}\binom{mn-s-2}{r-1}\\
		& = \binom{n-t+1}{r+s+1}\binom{mn-s-1}{r}\\
		& \kern1cm \times\Biggl(\frac{ns-ts-rs-s^{2}+n-t-r-s}{n-t+1} + \frac{msr+mtr-sr-r}{mn-s-1}\Biggr)\\
		& = \binom{n-t+1}{r+s+1}\binom{mn-s-1}{r}\\
		& \kern1cm \times\Biggl(\frac{(mns+mn-mr-s^2-sr-2s-r+mtr-1)(n-t-s)}{(n-t+1)(mn-s-1)}\Biggr)\\
		& = (n-t-s)\binom{n-t+1}{r+s+1}\binom{mn-s-1}{r}\\
		& \kern1cm \times\Biggl(\frac{mr(t-1)+(mn-r-s-1)(s+1)}{(n-t+1)(mn-s-1)}\Biggr)\\
		& = (n-t-s)\binom{n-t+1}{r+s+1}\binom{mn-s-1}{r}\Biggl(\frac{r+s+1}{n-t+1} - \frac{mr}{mn-s-1}\Biggr)\\
		& = (n-t-s)\Biggl(\binom{n-t}{r+s}\binom{mn-s-1}{r} - m\binom{n-t+1}{r+s+1}\binom{mn-s-2}{r-1}\Biggr).
	\end{align*}
	
	If we plug this simplified form back into the previously obtained form for $\NPol_{m,n,t}(x,y)$, then we observe that the denominator ``$n-t-s$'' of the outer sum has cancelled.  Moreover, setting $s=n-t$ in the outer sum yields exactly the additional summand $(xy)^{n-t}$.  We thus get:
	\begin{align*}
		\NPol_{m,n,t}(x,y) & = \sum_{s=0}^{n-t}\sum_{r=0}^{n-t-s}\Biggl(\binom{n-t}{r+s}\binom{mn-s-1}{r}\\
		& \kern3cm - m\binom{n-t+1}{r+s+1}\binom{mn-s-2}{r-1}\Biggr)x^{r+s}y^{s}.
	\end{align*}
	If we exchange the two sums and replace $s$ by $b$ and $r+s$ by $a$, we find:
	\begin{align*}
		\NPol_{m,n,t}(x,y) & = \sum_{a=0}^{n-t}\sum_{b=0}^{a}\Biggl(\binom{n-t}{a}\binom{mn-b-1}{a-b}\\
		& \kern3cm - m\binom{n-t+1}{a+1}\binom{mn-b-2}{a-b-1}\Biggr)x^{a}y^{b}.\qedhere
	\end{align*}
\end{proof}

\subsection{Proof of Proposition~\ref{prop:m_dyck_mvalley}}
\label{app:m_dyck_mvalley}

\aballform*
\begin{proof}[Detailed proof]
	We apply Lemma~\ref{lem:lagrange_burmann} to the functions $h(z)=\frac{z}{(z+1)^{m}(xz+1)}$ and $q_{\bullet}(z)=(z+1)^{m}(xz+1)$ and obtain the result:
	\begin{align*}
		\APol_{m,n,1}(x,1)
		& = \frac{1}{n-1}\langle z^{-1}\rangle \Biggl((z+1)^{m-1}\Bigl(m(xz+1) + x(z+1)\Bigr)\\
			& \kern5cm \times\left(\frac{z}{(z+1)^{m}(xz+1)}\right)^{-(n-1)}\Biggr)\\
		& = \frac{1}{n-1}\langle z^{n-2}\rangle \Biggl(m(z+1)^{mn-1}(xz+1)^{n} + x(z+1)^{mn}(xz+1)^{n-1}\Biggr)\\
		& = \frac{1}{n-1}\langle z^{n-2}\rangle \Biggl(m\sum_{k=0}^{n}\sum_{\ell=0}^{mn}\binom{n}{k}\binom{mn-1}{\ell}x^{k}z^{k+\ell}\\
			& \kern3cm + \sum_{k=0}^{n-1}\sum_{\ell=0}^{mn}\binom{n-1}{k}\binom{mn}{\ell}x^{k+1}z^{k+\ell}\Bigr)\Biggr)\\
		& = \frac{1}{n-1}\Biggl(m\sum_{r=0}^{n-2}\binom{n}{r}\binom{mn-1}{n-r-2}x^{r} + \sum_{r=0}^{n-2}\binom{n-1}{r}\binom{mn}{n-r-2}x^{r+1}\Biggr)\\
		& = \frac{1}{n-1}\sum_{r=0}^{n-1}\Biggl(m\binom{n}{r}\binom{mn-1}{n-r-2} + \binom{n-1}{r-1}\binom{mn}{n-r-1}\Biggr)x^{r}\\
		& = \frac{1}{n-1}\sum_{r=0}^{n-1}\Biggl(\frac{m(n-r-1)}{mn} + \frac{r}{n}\Biggr)\binom{n}{r}\binom{mn}{n-r-1}x^{r}\\
		& = \sum_{r=0}^{n-1}\frac{1}{n}\binom{n}{r}\binom{mn}{n-r-1}x^{r}.\qedhere
	\end{align*}
\end{proof}

\subsection{Proof of Proposition~\ref{prop:m_dyck_mvalley_noret}}
\label{app:m_dyck_mvalley_noret}

\aposform*
\begin{proof}[Detailed proof]
	We apply Lemma~\ref{lem:lagrange_burmann} to $h(z)=\frac{z}{(z+1)^{m}(xz+1)}$ and $q_{+,m}(z) = (z+1)^{m}$.  We get:
	\begin{align*}
		\APol_{m,n,1}(x,0)
		& = \frac{1}{n-1}\langle z^{-1}\rangle\Biggl(m(z+1)^{m-1}\left(\frac{z}{(z+1)^{m}(xz+1)}\right)^{-(n-1)}\Biggr)\\
		& = \frac{1}{n-1}\langle z^{n-2}\rangle\Biggl(m(z+1)^{mn-1}(xz+1)^{n-1}\Biggr)\\
		& = \frac{1}{n-1}\langle z^{n-2}\rangle\Biggl(m\sum_{k=0}^{n-1}\sum_{\ell=0}^{mn-1}\binom{n-1}{k}\binom{mn-1}{\ell}x^{k}z^{k+\ell}\Biggr)\\
		& = \sum_{r=0}^{n-1}\frac{m}{n-1}\binom{n-1}{r}\binom{mn-1}{n-r-2}x^{r}.\qedhere
	\end{align*}
\end{proof}

\subsection{Proof of Proposition~\ref{prop:m_dyck_mvalley_noret_var}}
\label{app:m_dyck_mvalley_noret_var}

\bposform*    
\begin{proof}[Detailed proof]
	We apply Lemma~\ref{lem:lagrange_burmann} to $h(z)=\frac{z}{(z+1)^{m}(xz+1)}$ and $q_{+,i}(z) = (z+1)^{m-1}(xz+1)$.  We get:
	\begin{align*}
		\BPol_{m,n,1}(x,0)
		& = \frac{1}{n-1}\langle z^{-1}\rangle\Biggl((z+1)^{m-2}\Bigl((m-1)(xz+1) + x(z+1)\Bigr)\\
			& \kern5cm \times\left(\frac{z}{(z+1)^{m}(xz+1)}\right)^{-(n-1)}\Biggr)\\
		& = \frac{1}{n-1}\langle z^{n-2}\rangle\Biggl((m-1)(z+1)^{mn-2}(xz+1)^{n} + x(z+1)^{mn-1}(xz+1)^{n-1}\Biggr)\\
		& = \frac{1}{n-1}\langle z^{n-2}\rangle\Biggl((m-1)\sum_{k=0}^{mn-2}\sum_{\ell=0}^{n}\binom{mn-2}{k}\binom{n}{\ell}x^{\ell}z^{k+\ell}\\
			& \kern3cm + \sum_{k=0}^{mn-1}\sum_{\ell=0}^{n-1}\binom{mn-1}{k}\binom{n-1}{\ell}z^{k+\ell}x^{\ell+1}\Biggr)\\
		& = \frac{1}{n-1}\sum_{r=0}^{n-1}x^{r}\Biggl((m-1)\binom{mn-2}{n-r-2}\binom{n}{r} + \binom{mn-1}{n-r-1}\binom{n-1}{r-1}\Biggr)\\
		& = \frac{1}{n-1}\sum_{r=0}^{n-1}x^{r}\Biggl(\frac{(m-1)n}{n-r}\binom{mn-2}{n-r-2}\binom{n-1}{r} + \frac{r}{n-r}\binom{mn-1}{n-r-1}\binom{n-1}{r}\Biggr)\\
		& = \frac{1}{n-1}\sum_{r=0}^{n-1}\binom{n-1}{r}\binom{mn-2}{n-r-1}x^{r}\\
			& \kern3cm \times\Biggl(\frac{(m-1)n(n-r-1)}{(n-r)(mn-n+r)} + \frac{r(mn-1)}{(n-r)(mn-n+r)}\Biggr)\\
		& = \frac{1}{n-1}\sum_{r=0}^{n-1}\binom{n-1}{r}\binom{mn-2}{n-r-1}x^{r}\Biggl(\frac{(mn-n+r)(n-1)}{(n-r)(mn-n+r)}\Biggr)\\
		& = \sum_{r=0}^{n-1}\frac{1}{n}\binom{n}{r}\binom{mn-2}{n-r-1}x^{r}.\qedhere
	\end{align*}
\end{proof}

\subsection{Proof of Theorem~\ref{thm:mnt_htriangle}}
\label{app:mnt_htriangle}

\aform*
\begin{proof}[Detailed proof]
	We consider $q_{m}(z) = (xz+1)^{t-1}(z+1)^{(m-1)(t-1)+m(s+1)}$.  Let us, for the record, explicitly state the derivative of $q_{m}(z)$:
	\begin{align*}
		q'_{m}(z)
		& = (xz+1)^{t-2}(z+1)^{(m-1)(t-1)+m(s+1)-1}\\
			& \kern1cm \times\Bigl((t-1)x(z+1) + \bigl((m-1)(t-1)+m(s+1)\bigr)(xz+1)\Bigr).
	\end{align*}
	In view of \eqref{eq:mdyck_mval_ret_gf_A}, we apply Lemma~\ref{lem:lagrange_burmann} to $h(z)=\frac{z}{(z+1)^{m}(xz+1)}$ and $q_{m}(z)$, which gives us:
	\begin{align*}
		\APol_{m,n,t} & (x,y) = \sum_{s=0}^{\infty}(xy)^{s}\langle z^{n-t-s}\rangle q_{m}\Bigl(H_{m}^{(1)}(x,1;z)\Bigr)\\
		& = (xy)^{n-t} + \sum_{s=0}^{n-t-1}\frac{(xy)^{s}}{n-t-s}\langle z^{-1}\rangle\Biggl(q'_{m}(z)\Bigl(h(z)\Bigr)^{-(n-t-s)}\Biggr)\\
		& = (xy)^{n-t} + \sum_{s=0}^{n-t-1}\frac{(xy)^{s}}{n-t-s}\langle z^{n-t-s-1}\rangle\Biggl(q'_{m}(z)(z+1)^{m(n-t-s)}(xz+1)^{n-t-s}\Biggr)\\
		& = (xy)^{n-t} + \sum_{s=0}^{n-t-1}\frac{(xy)^{s}}{n-t-s}\langle z^{n-t-s-1}\rangle\Biggl((t-1)x(xz+1)^{n-s-2}(z+1)^{mn-t+1}\\
			& \kern2cm + \Bigl((m-1)(t-1)+m(s+1)\Bigr)(xz+1)^{n-s-1}(z+1)^{mn-t}\Biggr)\\
		& = (xy)^{n-t} + \sum_{s=0}^{n-t-1}\frac{(xy)^{s}}{n-t-s}\langle z^{n-t-s-1}\rangle\\
			& \kern2cm \Biggl((t-1)\sum_{k=0}^{n-s-2}\sum_{\ell=0}^{mn-t+1}\binom{n-s-2}{k}\binom{mn-t+1}{\ell}x^{k+1}z^{k+\ell}\\
			& \kern2.5cm + \Bigl(mt+ms-t+1\Bigr)\sum_{k=0}^{n-s-1}\sum_{\ell=0}^{mn-t}\binom{n-s-1}{k}\binom{mn-t}{\ell}x^{k}z^{k+\ell}\Biggr)\\
		& = (xy)^{n-t} + \sum_{s=0}^{n-t-1}\frac{(xy)^{s}}{n-t-s}\Biggl((t-1)\sum_{r=0}^{n-t-s-1}\binom{n-s-2}{r}\binom{mn-t+1}{n-t-s-1-r}x^{r+1}\\
			& \kern2cm + \Bigl(mt+ms-t+1\Bigr)\sum_{r=0}^{n-t-s-1}\binom{n-s-1}{r}\binom{mn-t}{n-t-s-1-r}x^{r}\Biggr)\\
		& = (xy)^{n-t} + \sum_{s=0}^{n-t-1}\frac{(xy)^{s}}{n-t-s}\sum_{r=0}^{n-t-s}\Biggl((t-1)\binom{n-s-2}{r-1}\binom{mn-t+1}{n-t-s-r}\\
			& \kern2cm + \Bigl(mt+ms-t+1\Bigr)\binom{n-s-1}{r}\binom{mn-t}{n-t-s-r-1}\Biggr)x^{r}\\
		& = (xy)^{n-t} + \sum_{s=0}^{n-t-1}\frac{(xy)^{s}}{n-t-s}\sum_{r=0}^{n-t-s}\binom{n-s-1}{r}\binom{mn-t+1}{n-t-s-r}x^{r}\\
			& \kern2cm \times\Biggl(\frac{(t-1)r}{n-s-1} + \frac{(mt+ms-t+1)(n-t-s-r)}{mn-t+1}\Biggr)\\
		& = (xy)^{n-t} + \sum_{s=0}^{n-t-1}\frac{(xy)^{s}}{(n-t-s)(n-s-1)(mn-t+1)}\\
			& \kern2cm \times\sum_{r=0}^{n-t-s}\binom{n-s-1}{r}\binom{mn-t+1}{n-t-s-r}x^{r}(n-t-s)\\
			& \kern2.5cm \times\Bigl((ms+mt-t+1)(n-s-1) - r(ms+m-t+1)\Bigr).
	\end{align*}
	At this point, we notice that the running denominator ``$n-t-s$'' cancels.  The separate summand ``$(xy)^{n-t}$'' may be integrated into the sum, because it agrees with the term contributed by $s=n-t$.  We thus find
	\begin{align*}
		\APol_{m,n,t}(x,y) & = \sum_{s=0}^{n-t}\frac{(xy)^{s}}{(n-s-1)(mn-t+1)}\sum_{r=0}^{n-t-s}\binom{n-s-1}{r}\binom{mn-t+1}{n-t-s-r}x^{r}\\
			& \kern2cm \times\Bigl((ms+mt-t+1)(n-s-1) - r(ms+m-t+1)\Bigr)\\
		& = \sum_{s=0}^{n-t}\sum_{b=s}^{n-t}x^{b}y^{s}\Biggl(\frac{ms+mt-t+1}{mn-t+1}\binom{n-s-1}{b-s}\binom{mn-t+1}{n-t-b}\\
			& \kern2cm - \frac{ms+m-t+1}{mn-t+1}\binom{n-s-2}{b-s-1}\binom{mn-t+1}{n-t-b}\Biggr)\\
		& = \sum_{s=0}^{n-t}\sum_{h=0}^{n-t-s}x^{n-t-h}y^{s}\Biggl(\frac{ms+mt-t+1}{mn-t+1}\binom{n-s-1}{t+h-1}\binom{mn-t+1}{h}\\
			& \kern2cm - \frac{ms+m-t+1}{mn-t+1}\binom{n-s-2}{t+h-1}\binom{mn-t+1}{h}\Biggr).
	\end{align*}
	In this expression, we exchange the sums, replace $s$ by $b$ and $n-t-h$ by $a$.  We thus get
	\begin{align*}
		\APol_{m,n,t}(x,y) 
		& = \sum_{a=0}^{n-t}\sum_{b=0}^{a}x^{a}y^{b}\Biggl(\frac{mb+mt-t+1}{mn-t+1}\binom{n-b-1}{a-b}\binom{mn-t+1}{n-t-a}\\
			& \kern2cm - \frac{mb+m-t+1}{mn-t+1}\binom{n-b-2}{a-b-1}\binom{mn-t+1}{n-t-a}\Biggr)\\
		& = \sum_{a=0}^{n-t}\sum_{b=0}^{a}\binom{n-b-2}{a-b}\binom{mn-t+1}{n-t-a}x^{a}y^{b}\\
			& \kern2cm \times\Biggl(1 + \frac{m(b-n+1)(n-t-a)}{(mn-t+1)(n-a-1)}\Biggr)\\
		& = \sum_{a=0}^{n-t}\sum_{b=0}^{a}\Biggl(\binom{n-b-2}{a-b}\binom{mn-t+1}{n-t-a}\\
			& \kern4cm - m\binom{n-b-1}{a-b}\binom{mn-t}{n-t-a-1}\Biggr)x^{a}y^{b}.\qedhere
	\end{align*}
\end{proof}

\subsection{Proof of Theorem~\ref{thm:mnt_htriangle_var}}
\label{app:mnt_htriangle_var}

\bform*
\begin{proof}[Detailed proof]
	Throughout this proof we fix $i\in[m-1]$.  We consider $q_{i}(z)=(xz+1)^{s+t}(z+1)^{(s+t)(m-1)}$.  Let us explictly write down the derivative of $q_{i}(z)$:
	\begin{align*}
		q'_{i}(z) & = (s+t)(xz+1)^{s+t-1}(z+1)^{(s+t)(m-1)-1}\Bigl(x(z+1) + (m-1)(xz+1)\Bigr).
	\end{align*}
	In view of \eqref{eq:mdyck_mval_ret_gf_A}, we apply Lemma~\ref{lem:lagrange_burmann} to $h(z)=\frac{z}{(z+1)^{m}(xz+1)}$ and $q_{i}(z)$, which gives us:
	\begin{align*}
		\BPol_{m,n,t} & (x,y) = \sum_{s=0}^{\infty}y^{s}\langle z^{n-t-s}\rangle q_{i}\Bigl(H_{i}^{(1)}(x,1;z)\Bigr)\\
		& = y^{n-t} + \sum_{s=0}^{n-t-1}\frac{y^{s}}{n-t-s}\langle z^{-1}\rangle\Biggl(q'_{i}(z)\Bigl(h(z)\Bigr)^{-(n-t-s)}\Biggr)\\
		& = y^{n-t} + \sum_{s=0}^{n-t-1}\frac{y^{s}}{n-t-s}\langle z^{n-t-s-1}\rangle\Biggl(q'_{i}(z)\Bigl((z+1)^{m(n-t-s)}(xz+1)^{n-t-s}\Bigr)\Biggr)\\
		& = y^{n-t} + \sum_{s=0}^{n-t-1}\frac{(s+t)y^{s}}{n-t-s}\langle z^{n-t-s-1}\rangle\\
			& \kern3cm \Biggl(\Bigl(x(xz+1)^{n-1}(z+1)^{mn-t-s}\\
			& \kern3.5cm + (m-1)(xz+1)^{n}(z+1)^{mn-t-s-1}\Bigr)\Biggr)\\
		& = y^{n-t} + \sum_{s=0}^{n-t-1}\frac{(s+t)y^{s}}{n-t-s}\langle z^{n-t-s-1}\rangle\\
			& \kern3cm \Biggl(\sum_{k=0}^{n-1}\sum_{\ell=0}^{mn-t-s}\binom{n-1}{k}\binom{mn-t-s}{\ell}x^{k+1}z^{k+\ell}\\
			& \kern3.5cm + (m-1)\sum_{k=0}^{n}\sum_{\ell=0}^{mn-t-s-1}\binom{n}{k}\binom{mn-t-s-1}{\ell}x^{k}z^{k+\ell}\Bigr)\Biggr)\\
		& = y^{n-t} + \sum_{s=0}^{n-t-1}\frac{(s+t)y^{s}}{n-t-s}\Biggl(\sum_{r=0}^{n-t-s-1}\binom{n-1}{r}\binom{mn-t-s}{n-t-s-r-1}x^{r+1}\\
			& \kern3.5cm + (m-1)\sum_{r=0}^{n-t-s-1}\binom{n}{r}\binom{mn-t-s-1}{n-t-s-r-1}x^{r}\Biggr)\\
		& = y^{n-t} + \sum_{s=0}^{n-t-1}\frac{(s+t)y^{s}}{n-t-s}\Biggl(\sum_{r=0}^{n-t-s-1}\binom{n-1}{r}\binom{mn-t-s}{mn-n+r+1}x^{r+1}\\
			& \kern3.5cm + (m-1)\sum_{r=0}^{n-t-s-1}\binom{n}{r}\binom{mn-t-s-1}{mn-n+r}x^{r}\Biggr)\\
		& = y^{n-t} + \sum_{s=0}^{n-t-1}\frac{(s+t)y^{s}}{n-t-s}\sum_{r=0}^{n-t-s}x^{r}\Biggl(\binom{n-1}{r-1}\binom{mn-t-s}{mn-n+r}\\
			& \kern3.5cm + (m-1)\binom{n}{r}\binom{mn-t-s-1}{mn-n+r}\Biggr)\\
		& = y^{n-t} + \sum_{s=0}^{n-t-1}\frac{(s+t)y^{s}}{n-t-s}\sum_{r=0}^{n-t-s}\binom{n}{r}\binom{mn-t-s}{mn-n+r}x^{r}\\
			& \kern3.5cm \times\Biggl(\frac{r}{n} + \frac{(m-1)(n-t-s-r)}{mn-t-s}\Biggr)\\
		& = y^{n-t} + \sum_{s=0}^{n-t-1}\frac{(s+t)y^{s}}{n-t-s}\sum_{r=0}^{n-t-s}\binom{n}{r}\binom{mn-t-s}{mn-n+r}x^{r}\Biggl(\frac{(mn+r-n)(n-t-s)}{n(mn-t-s)}\Biggr).
	\end{align*}
	At this point, once again, the running denominator ``$n-t-s$'' cancels.  As before, we may also integrate the separate summand $y^{n-t}$ into the sum.  We thus get
	\begin{align*}
		\BPol_{m,n,t}(x,y) & = \sum_{s=0}^{n-t}\sum_{r=0}^{n-t-s}\frac{s+t}{n}x^{r}y^{s}\binom{n}{r}\binom{mn-t-s}{mn-n+r}\frac{mn-n+r}{mn-t-s}\\
		& = \sum_{s=0}^{n-t}\sum_{r=0}^{n-t-s}\frac{s+t}{n}x^{r}y^{s}\binom{n}{r}\binom{mn-t-s-1}{mn-n+r-1}\\
		& = \sum_{r=0}^{n-t}\sum_{s=0}^{n-t-r}\frac{s+t}{n}x^{r}y^{s}\binom{n}{r}\binom{mn-t-s-1}{n-t-r-s}.\qedhere
	\end{align*}
\end{proof}


\begin{bibdiv}
\begin{biblist}

\normalfont

\bib{armstrong09generalized}{article}{
      author={Armstrong, Drew},
       title={Generalized noncrossing partitions and combinatorics of {C}oxeter groups},
        date={2009},
     journal={Mem. Amer. Math. Soc.},
      volume={202},
      number={949},
       pages={x+159},
}

\bib{Ath_RGCN05}{article}{
      author={Athanasiadis, Christos.~A.},
       title={On a refinement of the generalized {C}atalan numbers for {W}eyl groups},
        date={2005},
     journal={{T}rans. {A}mer. {M}ath. {S}oc.},
      volume={357},
       pages={179\ndash 196},
}

\bib{athanasiadis07some}{article}{
      author={Athanasiadis, Christos~A.},
       title={On some enumerative aspects of generalized associahedra},
        date={2007},
     journal={European J. Combin.},
      volume={28},
      number={4},
       pages={1208\ndash 1215},
}

\bib{atz-pc-06}{article}{
      author={Athanasiadis, Christos~A.},
      author={Tzanaki, Eleni},
       title={On the enumeration of positive cells in generalized cluster complexes and {C}atalan hyperplane arrangements},
        date={2006},
     journal={J. Algebraic Combin.},
      volume={23},
      number={4},
       pages={355\ndash 375},
}

\bib{beck18combinatorial}{book}{
      author={Beck, Matthias},
      author={Sanyal, Raman},
       title={Combinatorial reciprocity theorems},
      series={Graduate Studies in Mathematics},
   publisher={American Mathematical Society, Providence, RI},
        date={2018},
      volume={195},
        note={An invitation to enumerative geometric combinatorics},
}

\bib{burstein20distribution}{article}{
      author={Burstein, Alexander},
       title={Distribution of peak heights modulo $k$ and double descents on $k$-{D}yck paths},
        date={2020},
     journal={arXiv:2009.00760},
}

\bib{ceballos20steep}{article}{
      author={Ceballos, Cesar},
      author={Fang, Wenjie},
      author={M\"{u}hle, Henri},
       title={The steep-bounce zeta map in {P}arabolic {C}ataland},
        date={2020},
     journal={J. Combin. Theory Ser. A},
      volume={172},
       pages={105210, 59 pages},
}

\bib{ceballos21fh}{article}{
      author={Ceballos, Cesar},
      author={M{\"u}hle, Henri},
       title={${F}$- and ${H}$-triangles for $\nu$-associahedra},
        date={2021},
     journal={Comb. Theory},
      volume={2},
      number={2},
       pages={Art. 3, 26 pages},
}

\bib{ceballos20nu}{article}{
      author={Ceballos, Cesar},
      author={Padrol, Arnau},
      author={Sarmiento, Camilo},
       title={The {$\nu$}-{T}amari lattice via {$\nu$}-trees, {$\nu$}-bracket vectors, and subword complexes},
        date={2020},
     journal={Electron. J. Combin.},
      volume={27},
      number={1},
       pages={Paper No. 1.14, 31 pages},
}

\bib{chapoton04enumerative}{article}{
      author={Chapoton, Fr\'{e}d\'{e}ric},
       title={Enumerative properties of generalized associahedra},
        date={2004/05},
     journal={S\'{e}m. Lothar. Combin.},
      volume={51},
       pages={Art. B51b, 16 pages},
}

\bib{chapoton06sur}{article}{
      author={Chapoton, Fr{\'e}d{\'e}ric},
       title={Sur le nombre de r\'{e}flexions pleines dans les groupes de {C}oxeter finis},
        date={2006},
     journal={Bull. Belg. Math. Soc. Simon Stevin},
      volume={13},
      number={4},
       pages={585\ndash 596},
}

\bib{cigler87some}{article}{
      author={Cigler, Johann},
       title={Some remarks on {C}atalan families},
        date={1987},
     journal={European J. Combin.},
      volume={8},
      number={3},
       pages={261\ndash 267},
}


{collection.article}{
      author={Krattenthaler, Christian},
       title={The {$F$}-triangle of the generalised cluster complex},
        date={2006},
   booktitle={{Topics in Discrete Mathematics}},
      editor={Klazar, Martin},
      editor={Kratochv{\' i}l, Jan},
      editor={Loebl, Martin},
      editor={Matou{\v s}ek, Ji{\v r}{\' i}},
      editor={Valtr, Pavel},
      editor={Thomas, Robin},
      series={Algorithms Combin.},
      volume={26},
   publisher={Springer},
     address={Berlin},
       pages={93\ndash 126},
}


\bib{duchon00enumeration}{article}{
      author={Duchon, Philippe},
       title={On the enumeration and generation of generalized {D}yck words},
        date={2000},
     journal={Discrete Math.},
      volume={225},
      number={1-3},
       pages={121\ndash 135},
}	

\bib{edelman80chain}{article}{
      author={Edelman, Paul~H.},
       title={Chain enumeration and noncrossing partitions},
        date={1980},
     journal={Discrete Math.},
      volume={31},
      number={2},
       pages={171\ndash 180},
}

\bib{ehrhart62sur}{article}{
      author={Ehrhart, Eug\`ene},
       title={Sur les poly\`edres rationnels homoth\'{e}tiques \`a {$n$} dimensions},
        date={1962},
     journal={C. R. Acad. Sci. Paris},
      volume={254},
       pages={616\ndash 618},
}

\bib{fang21consecutive}{article}{
      author={Fang, Wenjie},
      author={M\"{u}hle, Henri},
      author={Novelli, Jean-Christophe},
       title={A consecutive {L}ehmer code for parabolic quotients of the symmetric group},
        date={2021},
     journal={Electron. J. Combin.},
      volume={28},
      number={3},
       pages={Paper No. 3.53, 28 pages},
}

\bib{fang21parabolic}{article}{
      author={Fang, Wenjie},
      author={M\"{u}hle, Henri},
      author={Novelli, Jean-Christophe},
       title={Parabolic {T}amari lattices in linear type ${B}$},
        date={2021},
     journal={arXiv:2112.13400},
}

\bib{fr-gcccc-05}{article}{
      author={Fomin, Sergey},
      author={Reading, Nathan},
       title={{Generalized cluster complexes and {C}oxeter combinatorics}},
        date={2005},
     journal={Int. Math. Res. Not.},
      volume={44},
       pages={2709\ndash 2757},
}

\bib{fomin03ysystems}{article}{
      author={Fomin, Sergey},
      author={Zelevinsky, Andrei},
       title={{$Y$}-systems and generalized associahedra},
        date={2003},
     journal={Ann. of Math. (2)},
      volume={158},
      number={3},
       pages={977\ndash 1018},
}

\bib{henrici74applied}{book}{
      author={Henrici, Peter},
       title={Applied and computational complex analysis},
      series={Pure and Applied Mathematics},
   publisher={Wiley-Interscience [John Wiley \& Sons], New York-London-Sydney},
        date={1974},
        note={Volume 1: Power series---integration---conformal mapping---location of zeros},
}

\bib{heuberger22enumeration}{article}{
      author={Heuberger, Clemens},
      author={Selkirk, Sarah~J.},
      author={Wagner, Stephan},
       title={Enumeration of generalized {D}yck paths based on the height of down-steps modulo $k$},
        date={2022},
     journal={arXiv:2204.14023},
}

\bib{krattenthaler89counting}{article}{
      author={Krattenthaler, Christian},
       title={Counting lattice paths with a linear boundary. {II}. {$q$}-ballot and {$q$}-{C}atalan numbers},
        date={1989},
     journal={\"{O}sterreich. Akad. Wiss. Math.-Natur. Kl. Sitzungsber. II},
      volume={198},
      number={4-7},
       pages={171\ndash 199},
}

\bib{krattenthaler05mtriangle}{article}{
      author={Krattenthaler, Christian},
       title={The {$M$}-triangle of generalised non-crossing partitions for the types {$E_7$} and {$E_8$}},
        date={2005/07},
     journal={S\'{e}m. Lothar. Combin.},
      volume={54},
       pages={Art. B541, 34 pages},
}

\bib{krattenthaler06ftriangle}{collection.article}{
      author={Krattenthaler, Christian},
       title={The {$F$}-triangle of the generalised cluster complex},
        date={2006},
   booktitle={{Topics in Discrete Mathematics}},
      editor={Klazar, Martin},
      editor={Kratochv{\' i}l, Jan},
      editor={Loebl, Martin},
      editor={Matou{\v s}ek, Ji{\v r}{\' i}},
      editor={Valtr, Pavel},
      editor={Thomas, Robin},
      series={Algorithms Combin.},
      volume={26},
   publisher={Springer},
     address={Berlin},
       pages={93\ndash 126},
}

\bib{krattenthaler15lattice}{collection.article}{
      author={Krattenthaler, Christian},
       title={{Lattice Path Enumeration}},
        date={2015},
   booktitle={{Handbook of Enumerative Combinatorics}},
      editor={B{\'o}na, Mikl{\'o}s},
      series={Discrete Mathematics and Its Applications},
      volume={87},
   publisher={CRC Press},
     address={Boca Raton-London-New York},
       pages={589\ndash 678},
}

\bib{krattenthaler22rank}{article}{
      author={Krattenthaler, Christian},
      author={M{\"u}hle, Henri},
       title={The rank enumeration of certain parabolic non-crossing partitions},
        date={2022},
     journal={Algebr. Comb.},
      volume={5},
      number={3},
       pages={437\ndash 468},
}

\bib{krattenthaler10decomposition}{article}{
      author={Krattenthaler, Christian},
      author={M\"{u}ller, Thomas~W.},
       title={Decomposition numbers for finite {C}oxeter groups and generalised non-crossing partitions},
        date={2010},
     journal={Trans. Amer. Math. Soc.},
      volume={362},
      number={5},
       pages={2723\ndash 2787},
}

\bib{muehle20ballot}{article}{
      author={M\"{u}hle, Henri},
       title={Ballot-noncrossing partitions},
        date={2020},
     journal={S\'{e}m. Lothar. Combin.},
      volume={82B},
       pages={Art. 7, 12 pages},
}

\bib{muehle21noncrossing}{article}{
      author={M\"{u}hle, Henri},
       title={Noncrossing arc diagrams, {T}amari lattices, and parabolic quotients of the symmetric group},
        date={2021},
     journal={Ann. Comb.},
      volume={25},
      number={2},
       pages={307\ndash 344},
}

\bib{muehle19tamari}{article}{
      author={M\"{u}hle, Henri},
      author={Williams, Nathan},
       title={Tamari lattices for parabolic quotients of the symmetric group},
        date={2019},
     journal={Electron. J. Combin.},
      volume={26},
      number={4},
       pages={Paper No. 4.34, 28 pages},
}

\bib{stanley74combinatorial}{article}{
      author={Stanley, Richard~P.},
       title={Combinatorial reciprocity theorems},
        date={1974},
     journal={Advances in Math.},
      volume={14},
       pages={194\ndash 253},
}

\bib{Stan_HAPFTI98}{collection.article}{
      author={Stanley, Richard~P.},
       title={Hyperplane arrangements, parking functions and tree inversions},
        date={1998},
   booktitle={Mathematical essays in honor of {G}ian-{C}arlo {R}ota ({C}ambridge, {MA}, 1996)},
      editor={Sagan, Bruce~E.},
      editor={Stanley, Richard~P.},
      series={Progress in Mathematics},
   publisher={Birkh{\"a}user},
     address={Boston, MA},
      volume={161},
       pages={359\ndash 375},
}

\bib{thiel14htriangle}{article}{
      author={Thiel, Marko},
       title={On the {$H$}-triangle of generalised nonnesting partitions},
        date={2014},
     journal={European J. Combin.},
      volume={39},
       pages={244\ndash 255},
}

\bib{tz-pdgcc-06}{article}{
      author={Tzanaki, Eleni},
       title={Polygon dissections and some generalizations of cluster complexes},
        date={2006},
     journal={J. Combin. Theory Ser. A},
      volume={113},
       pages={1189\ndash 1198},
}

\bib{tzanaki08faces}{article}{
      author={Tzanaki, Eleni},
       title={Faces of generalized cluster complexes and noncrossing partitions},
        date={2008},
     journal={SIAM J. Discrete Math.},
      volume={22},
      number={1},
       pages={15\ndash 30},
}

\bib{bell21schroder}{article}{
      author={von Bell, Matias},
      author={Yip, Martha},
       title={Schr\"{o}der combinatorics and {$\nu$}-associahedra},
        date={2021},
        ISSN={0195-6698},
     journal={European J. Combin.},
      volume={98},
       pages={Paper No. 103415, 18 pages},
}

\bib{williams13cataland}{book}{
      author={Williams, Nathan},
       title={Cataland},
   publisher={ProQuest LLC, Ann Arbor, MI},
        date={2013},
        note={Thesis (Ph.D.)--University of Minnesota},
}

\end{biblist}
\end{bibdiv}
\end{document}